\documentclass[12pt]{amsart}
\usepackage[mathscr]{eucal}
\usepackage{amsmath,amsfonts}
\usepackage[all]{xy}
\usepackage{color}

\parskip=\smallskipamount

\newtheorem{theorem}{Theorem}[section]
\newtheorem{proposition}[theorem]{Proposition}
\newtheorem{corollary}[theorem]{Corollary}
\newtheorem{lemma}[theorem]{Lemma}

\theoremstyle{definition}

\newtheorem{examples}[theorem]{Examples}

\newtheorem{remark}[theorem]{Remark}
\newtheorem{remarks}[theorem]{Remarks}

\makeatletter
\@addtoreset{equation}{section}
\makeatother

\newcommand{\CC}{{\mathbb C}}
\newcommand{\NN}{{\mathbb N}}

\newcommand{\RR}{{\mathbb R}}

\newcommand{\cB}{{\mathcal B}}
\newcommand{\cC}{{\mathcal C}}

\newcommand{\cE}{{\mathcal E}}

\newcommand{\cH}{{\mathcal H}}

\newcommand{\cL}{{\mathcal L}}

\newcommand{\cP}{{\mathcal P}}
\newcommand{\cQ}{{\mathcal Q}}

\newcommand{\cS}{{\mathcal S}}
\newcommand{\cT}{{\mathcal T}}

\newcommand{\bA}{{\mathbf A}}
\newcommand{\bB}{{\mathbf B}}

\newcommand{\Ra}{\Rightarrow}
\newcommand{\ran}{\operatorname{Ran}}
\newcommand{\rank}{\operatorname{rank}}
\newcommand{\sgn}{\operatorname{sgn}}
\newcommand{\ra}{\rightarrow}

\newcommand{\tr}{\operatorname{tr}}
\let\phi=\varphi
\newcommand{\iac}{\mathrm{i}}
\renewcommand{\ker}{\operatorname{Null}}

\renewcommand{\Re}{\operatorname{Re}}

\newcommand{\lin}{\operatorname{Lin}}
\newcommand{\nr}[1]{\vspace{0.1ex}\noindent\hspace*{12mm}\llap{\textup{(#1)}}}

\title[Interpolation for Completely Positive Maps]{An Interpolation Problem for 
Completely Positive Maps on Matrix Algebras: Solvability and Parametrisation}

\author{C\u alin Ambrozie}\thanks{The first named author's research 
was supported by the grants IAA 100190903 of GAAV (RVO:67985840), and 
CNCSIS UEFISCDI PN-II-ID-PCE-2011-3-0119}

\address{Institute of Mathematics - Romanian Academy, PO Box 1-764, RO 014700 
Bucharest, Romania \emph{and} 
Institute of Mathematics of the Czech Academy, Zitna 25, 11567 Prague 1, Czech 
Republic}
\email{calin.ambrozie@imar.ro}

\date{\today}

\author{Aurelian Gheondea}\thanks{The second named author's research 
supported by a grant of the Romanian 
National Authority for Scientific Research, CNCSIS Ð UEFISCDI, project number
PN-II-ID-PCE-2011-3-0119.}
\address{Department of Mathematics, Bilkent University, 06800 Bilkent, Ankara, 
Turkey, \emph{and} Institutul de Matematic\u a al Academiei Rom\^ane, C.P.\ 
1-764, 014700 Bucure\c sti, Rom\^ania} 
\email{aurelian@fen.bilkent.edu.tr \textrm{and} A.Gheondea@imar.ro} 



\begin{document}
\maketitle

\begin{abstract}
We present certain existence criteria and parameterisations for an interpolation
problem for
completely positive maps that take given matrices from a 
finite set into prescribed matrices. Our approach uses
density matrices associated to linear functionals on $*$-subspaces of matrices,
inspired by the Smith-Ward linear functional and Arveson's Hahn-Banach type 
Theorem. We perform a careful investigation on the intricate 
relation between the positivity of the density matrix and the positivity of the 
corresponding 
linear functional. A necessary and sufficient condition for the existence 
of solutions and a parametrisation of the set of all solutions of the interpolation 
problem in terms of a closed and convex set of an affine space are obtained. 
Other linear affine restrictions, like trace preserving, can be included as 
well, hence covering applications to quantum channels that 
yield certain quantum states at prescribed quantum states. \end{abstract}

\section{Introduction}

The most general mathematical model of state changes in quantum mechanics,
including the evolution of an open system or the state change due to a 
measurement, is provided by the concept of "quantum operation", that is, a linear, 
completely positive, trace preserving map, cf.\ K.~Kraus \cite{Kra71} and
\cite{Kra83}.
Letting $M_n$ denote the unital $C^*$-algebra 
of all $n\times n$ complex matrices, recall that
a matrix $A\in M_n$ is \emph{positive semidefinite} if all its 
principal determinants are nonnegative. 
A linear map $\phi\colon M_n\ra M_k$ is \emph{completely positive} if, 
for all $m\in\NN$, the 
linear map $I_m\otimes \phi\colon M_m\otimes M_n\ra M_m\otimes M_k$ is 
\emph{positive}, 
in the sense that it maps any positive semidefinite element from $M_m\otimes M_n$ 
into a positive semidefinite element in $M_m\otimes M_k$. By $\mathrm{CP}
(M_n,M_k)$ 
we denote the cone of all completely positive maps $\phi\colon M_n\ra M_k$. An 
equivalent notion, cf.\ W.F.~Stinespring \cite{Stinespring}, is that of \emph{positive 
semidefinite} map $\phi$, 
that is, for all $m\in\NN$, all $h_1,\ldots,h_m\in \CC^n$, and all 
$A_1,\ldots,A_m\in M_n$, we have
\begin{equation}\label{e:psd} \sum_{i,j=1}^m \langle \phi(A_j^*A_i)h_j,h_i\rangle 
\geq 0.
\end{equation}
A \emph{quantum channel} is a completely positive linear 
map that is trace preserving.

A natural question related to these mathematical objects refers to 
finding a quantum channel that can take certain given quantum 
states from a finite list into some other prescribed quantum states.  
Thus, 
P.M. Alberti and A. Uhlmann \cite{AlbertiUhlmann} find a necessary and sufficient
condition for a pair of qubits (quantum states in $M_2$) 
to be mapped under the action of a quantum channel onto another given pair 
of qubits. For larger sets of pure states, the problem has been considered 
from many other perspectives, 
see A. Chefles, R. Jozsa, A. Winter \cite{CheflesJozsaWinter} and 
the bibliography cited there.

In this article we consider the following\medskip

\noindent\textbf{Interpolation Problem.} \emph{
Given matrices $A_\nu \in M_n$ and $B_\nu \in M_k$ for $\nu =1,\ldots ,N$, 
determine 
$\phi \in\mathrm{CP}(M_n ,M_k)$ subject to the conditions}
\begin{equation}\label{e:princ}\varphi(A_\nu) =B_\nu,\mbox{ \it for all }\nu=1,\ldots,N.
\end{equation}

The meaning of "determine" is rather vague so we have to make it clear: firstly, 
one should find necessary and/or sufficient conditions for the existence of such a 
solution $\phi$, secondly, one should find an explicit parametrisation of all solutions 
and, lastly, but not the least, one should find techniques (numerical, computational, 
etc.) to determine (approximate) solutions. Other conditions like trace 
preserving may be required as well, with direct applications to quantum operations.  
In this general formulation, these problems have been considered by 
C.-K.~Li and Y.-T. Poon in \cite{LiPoon}, 
where solutions have been obtained in case when the  
given states (more generally, Hermitian matrices) commute. More general
criteria for existence of solutions have been considered by 
Z.~Huang,  C.-K.~Li, E.~Poon,  and N.-S.~Sze in \cite{HLPZ}, while
T.~Heinosaari, M.A.~Jivulescu, D.~Reeb,  and M.M.~Wolf  obtain in
\cite{HJRW} other criteria of existence of solutions as well as 
techniques to approximate solutions in terms of semidefinite programming, 
in the sense of Y.~Nesterov and A.~Nemirovsky \cite{NesterovNemirovsky} and 
L.~Vanderberghe and S.~Boyd \cite{VanderbergheBoyd}.

The purpose of this article is to approach, from a general perspective, 
existence criteria and parametrisations of solutions of the Interpolation Problem. 
The solvability of the Interpolation Problem is characterised in Theorem~\ref{t:gen} 
from which an explicit parametrisation of the set of all solutions in terms of a closed
and convex set of an affine space follows.

In order to briefly describe our approach and results, let us denote
$\bA=(A_1,\ldots,A_N)$ and call it the \emph{input data} and, similarly, 
$\bB=(B_1,\ldots,B_N)$ and call it the \emph{output data}, as well as
\begin{equation}\label{e:data} 
\cC_{\bA,\bB}:=\{\phi\in\mathrm{CP}(M_n,M_k)\mid \phi(A_\nu)=B_\nu,
\mbox{ for all }\nu=1,\ldots,N\}.
\end{equation}
Clearly, the set $\cC_{\bA,\bB}$ is convex and closed, but it may or may 
not be compact. Since the maps $\phi\in\mathrm{CP}(M_n,M_k)$ are, by 
definition, linear, without loss of generality one can assume that the set 
$\{A_1,\ldots,A_N\}$ is linearly independent, otherwise some linear 
dependence conditions on the output data $\mathbf{B}$ are necessary. On the other 
hand, 
since any $\phi\in\mathrm{CP}(M_n,M_k)$ is Hermitian, in the 
sense that $\phi(A^*)=\phi(A)^*$ for all $A\in M_n$, 
it follows that, without loss of generality, 
one can assume that all matrices $A_1,\ldots,A_N,B_1,\ldots,B_N$ are Hermitian. 
In particular, 
letting $\cS_\bA$ denote the linear span of $A_1,\ldots,A_N$, it follows that 
$\cS_\bA$ is a $*$-subspace of $M_n$, that is, it is a linear subspace 
stable under taking adjoints, 
and then, letting $\phi_{\bA,\bB}\colon \cS_\bA\ra M_k$ be the linear map uniquely 
determined by the conditions
\begin{equation*}\phi_{\bA,\bB}(A_\nu)=B_\nu,\quad \nu=1,\ldots,N,
\end{equation*} it follows that any $\phi\in\mathrm{CP}(M_n,M_k)$ 
satisfying the constraints \eqref{e:princ} should necessarily be an extension of 
$\phi_{\bA,\bB}$. Thus, the Interpolation Problem may
require certain "positivity" properties
of the map $\phi_{\bA,\bB}$ but the $*$-subspace 
$\cS_\bA$ may not be linearly generated by $\cS_\bA^+$, the collection of its 
positive semidefinite matrices or, it may happen that $\cS_\bA$ may contain 
no nontrivial positive semidefinite matrix at all.

If the $*$-subspace $\cS_\bA$ contains the identity matrix $I_n$ (e.g.\
if we are interested in solutions $\phi$ that are unital, that is, $\phi(I_n)=I_k$), 
making 
it an \emph{operator system} \cite{Pau}, then $\cS_\bA$
is linearly generated by the cone of its positive semidefinite matrices. In this case, 
there is the celebrated
Arveson's Hahn-Banach Type Theorem \cite{Arveson}, see the 
equivalence of (a) and (d) in Theorem~\ref{t:sw}, saying that the 
Interpolation Problem has a solution if and only if $\phi_{\bA,\bB}$ has a certain 
"complete positivity" property. The original proof in \cite{Arveson} of this extension 
theorem was simplified by R.R.~Smith 
and J.D.~Ward \cite{SmithWard} who introduced a certain linear functional, 
that we call the \emph{Smith-Ward functional}, 
and then reducing the problem to finding positive extensions of it. 
In the case of the Interpolation Problem, letting 
$s_{\bA,\bB}$ denote the Smith-Ward linear functional associated to 
$\phi_{\bA,\bB}$, see \eqref{e:es}, one can go 
further and associate a "density matrix" $D_{\bA,\bB}$, see \eqref{e:dab}, 
to the Smith-Ward linear functional and then get, see Theorem~\ref{t:opsys}, 
that the solvability of the Interpolation Problem is equivalent with three assertions: 
firstly, with 
the complete positivity of $\phi_{\bA,\bB}$, secondly, with the positivity of the 
Smith-Ward linear functional $s_{\bA,\bB}$ and, finally, with the fact that the 
affine space $D_{\bA,\bB}+M_k\otimes\cS_\bA^\perp$
contains positive semidefinite matrices. 

On the other hand, a careful inspection of Theorem~\ref{t:opsys} shows that, 
actually, the Arveson's Hahn-Banach Theorem and the Smith-Ward 
linear functional $s_{\bA,\bB}$, seem to rather play the role of the hidden 
catalysts of the "reaction" but do not seem to play the major role. 
Indeed, the restrictive assumption that $\cS_\bA$ is 
generated by $\cS_\bA^+$ can be dropped and in Theorem~\ref{t:gen} 
we show that 
the solvability of the Interpolation Problem is equivalent with the fact that the affine 
subspace $D_{\bA,\bB}+M_k\otimes\cS_\bA^\perp$ 
contains positive semidefinite matrices, in the full generality. In addition,
Theorem~\ref{t:gen} yields, as a by-product, a parametrisation of the set of all 
solutions of the 
Interpolation Problem by the closed convex subset $\cP_{\bA,\bB}:=\{P\in (M_k
\otimes \cS_\bA^\perp)^\mathrm{h}\mid  P\geq -D_{\bA,\bB}\}$, through an 
affine isomorphism. In case 
$\cS_\bA$ is a $*$-subspace congruent to an operator system, then the 
parameterising convex subset $\cP_{\bA,\bB}$ is compact as well.

The density matrix $D_{\bA,\bB}$ plays the major role in our approach to 
the Interpolation Problem and it is a simple observation, see Remarks~\ref{r:texta}, 
that the positive semidefiniteness of $D_{\bA,\bB}$ is sufficient for the existence 
of solutions to the Interpolation Problem but, in general, this is not a necessary 
condition. We perform a careful investigation on this issue in Subsection~\ref{ss:dm}
and we provide examples and counter-examples illustrating the complexity of this 
phenomenon. In addition, in Theorem~\ref{t:mut} we show that, in case an 
operator system $\cS$ is generated by matrix units, then the density matrix of 
any positive 
linear functional on $\cS$ is positive semidefinite if and only if $\cS$ is an algebra. 
Therefore, it is the additional assumption that the $*$-subspace is an algebra that 
ensures the fact that the solvability of the Interpolation Problem is equivalent to the 
positive definiteness of $D_{\bA,\bB}$, even though exotic cases of $*$-subspaces 
that are not algebras but when this equivalence
happens may occur as well, see Examples~\ref{ex:pp}. 

In Subsection~\ref{ss:oid} we show that, if the input data $\bA$ is orthonormalised 
with respect to the Hilbert-Schmidt inner product,
then the density matrix is easily calculable as 
$D_{\bA,\bB}=\sum_{\nu=1}^N B_\nu^T\otimes A_\nu$ and this considerably
simplifies the criterion of solvability of the Interpolation Problem, see 
Theorem~\ref{t:ort}. Also, we observe that the Gram-Schmidt 
orthonormalisation does not affect the other assumptions.

Another important observation on the density matrix $D_{\bA,\bB}$ is that, one 
might think that it is the Choi matrix \cite{Choi} that plays the major role 
in getting criteria 
of existence of solutions of the Interpolation Problem, but this seems not to be the 
case: firstly, in order to define the Choi matrix, see Subsection~\ref{ss:ck}, 
we have to use all the matrix units, but the subspace $\cS_{\bA}$ might 
not contain any of them and, secondly, the Choi matrix does not relate well with the 
"action" of the linear map that it represents, while the density matrix does. 
Actually, once we explicitly show the relation between the density matrix and 
the Choi matrix of a given map $\phi\in\mathrm{CP}(M_n,M_k)$, see 
Proposition~\ref{p:trace}, we can define a "partial Choi matrix", 
see \eqref{e:choios}, for linear maps on subspaces.

Finally, in Subsection~\ref{ss:sip} we consider the Interpolation Problem for a 
single interpolation pair, that is, $N=1$, consisting of Hermitian matrices. 
By using techniques from indefinite inner
product spaces, e.g.\ see \cite{GLR05}, we derive criteria of existence of solutions
of the Interpolation Problem with only one operation element, get a 
necessary and sufficient condition of solvability in terms of the definiteness 
characteristics of the data, and estimate the minimal number of the operation 
elements of the solutions.

We thank Eduard Emelyanov for providing useful information 
on ordered vector spaces and especially for providing the bibliographical data 
on Kantorovich's Theorem. We also thank David Reeb for drawing our attention on 
\cite{HJRW}, soon after a first version of this manuscript has been circulated as a 
preprint, which also provided to us more information on literature on more or less 
special cases of the Interpolation Problem, that we were not aware of.

\section{Notation and Preliminary Results}\label{s:npr}

\subsection{The Choi Matrix and the Kraus Form.}\label{ss:ck}
For $n\in\NN$ let $\{e_i^{(n)}\}_{i=1}^n$ be the canonical basis of $\CC^n$.
As usually, the space $M_{n,k}$ of $n\times k$ matrices is identified 
with $\cB(\CC^k,\CC^n)$, the vector space of all linear transformations 
$\CC^k\ra \CC^n$.
For $n,k\in\NN$ we consider the matrix units 
$\{E_{l,i}^{(n,k)}\mid l=1,\ldots,n,\ 
i=1,\ldots,k\}\subset M_{n,k}$ of size $n\times k$, that is, $E_{l,i}^{(n,k)}$ is the $n\times k$
matrix with all entries $0$ except the $(l,i)$-th entry which is $1$.  
In case $n=k$, we denote simply $E^{(n)}_{l,i}=E^{(n,n)}_{l,i}$. Recall that $M_n$
is organized as a $C^*$-algebra in a natural way and hence, positive elements, that 
is, positive semidefinite matrices in $M_n$, are well defined.

Given a linear map $\phi\colon M_n\ra M_k$ define an $kn\times kn$ matrix 
$\Phi_\phi$ by
\begin{equation}\label{e:choi} \Phi_\phi =[\phi(E_{l,m}^{(n)})]_{l,m=1}^n.
\end{equation}
This transformation appears more or less explicitly at J.~de~Pillis \cite{Pillis}, 
A.~Jamio\l kowski \cite{Jamiolkowski}, R.D.~Hill \cite{Hill}, and M.D.~Choi 
\cite{Choi}.
In the following we describe more explicitly the transformation 
$\phi\mapsto \Phi_\phi$.
We use the lexicographic reindexing of 
$\{E_{l,i}^{(n,k)}\mid l=1,\ldots,n,\ 
i=1,\ldots,k\}$, more precisely
\begin{equation}\label{e:Es} 
\bigl(
E^{(n,k)}_{1,1},\ldots,E^{(n,k)}_{1,k},E^{(n,k)}_{2,1},\ldots,E^{(n,k)}_{2,k},
\ldots, E^{(n,k)}_{n,1},\ldots,E^{(n,k)}_{n,k}\bigr) 
=\bigl(\mathcal{E}_1,\mathcal{E}_2,....,\mathcal{E}_{nk} \bigr)
\end{equation}
An even more explicit form of this reindexing is the following
\begin{equation}\label{reindexing} \cE_r=E_{l,i}^{(n,k)}\mbox{ where }
r=(l-1)k+i,\mbox{ for all }l=1,\ldots,n,\ i=1,\ldots,k.
\end{equation}

The formula
\begin{equation}\label{e:prima} \phi_{(l-1)k+i,(m-1)k+j}
=\langle \phi(E_{l,m}^{(n)})e_{j}^{(k)},e_i^{(k)}\rangle,\quad i,j=1,\ldots,k,\
l,m=1,\ldots,n,
\end{equation} and its inverse
\begin{equation}\label{e:adoua} 
\phi(C)=\sum_{r,s=1}^{nk} \phi_{r,s} \cE_r^*C\cE_s,\quad C\in 
M_{n},
\end{equation} establish a linear and bijective correspondence 
\begin{equation}\label{e:defin}
\cB(M_n,M_k)\ni\phi\mapsto \Phi_\phi=[\phi_{r,s}]_{r,s=1}^{nk}\in M_{nk}.
\end{equation} 
The formulae \eqref{e:prima} and its inverse \eqref{e:adoua}
establish an affine and order preserving isomorphism
\begin{equation}\label{e:cpmp}
\mathrm{CP}(M_n,M_k)\ni\phi\mapsto\Phi_\phi\in M_{nk}^+.
\end{equation}
Given $\phi\in\cB(M_n,M_k)$ the matrix $\Phi_\phi$ as in \eqref{e:choi} is called
\emph{the Choi matrix} of $\phi$.

Let $\phi:M_n\rightarrow M_k$ be a completely 
positive map. Then, cf.\ K.~Kraus \cite{Kra71} and M.D.~Choi \cite{Choi}, 
there are $n\times k$ matrices 
$V_1,V_2,\ldots,V_m$ with $m\leq nk$ such that
\begin{equation}\label{e:kraus}
\phi(A)=V_1^*AV_1+V_2^*AV_2+\cdots +V_m^*AV_m \mbox{ for all }
A\in M_n. 
\end{equation}
The representation \eqref{e:kraus} is called the \emph{Kraus representation} of 
$\phi$ and the matrices $V_1,\dots,V_m$ are called the \emph{operation elements}. 
Note that the representation \eqref{e:kraus} of a given completely positive 
map $\phi$ is highly non-unique, not only with respect to its operation 
elements but also with respect to $m$, the number of these elements.
Concerning the minimal number of the operation elements in the Kraus 
form representation of completely positive maps on full matrix algebras, 
the following statement, which is implicit in the original article of M.D.~Choi 
\cite{Choi}, holds.

\begin{proposition}\label{p:minimal} Let $\phi\in\mathrm{CP}(M_n,M_k)$ 
be a completely 
positive map and let $\Phi$ be its Choi matrix. Then, $\rank(\Phi)$ is the 
minimal number of the operation elements of $\phi$.
\end{proposition}

\begin{proof} Let $V=[v_{l,i}]$, with 
$l=1,\ldots,n$ and $i=1,\ldots,k$, be a nontrivial $n\times k$ matrix such that 
$\phi(A)=V^*AV$ for all $A\in M_{n}$. Then, for all $l,m=1,\ldots,n$ and all 
$i,j=1,\dots,k$, by \eqref{e:prima} we have
\begin{equation} \phi_{(l-1)k+i,(m-1)k+j} 
= \langle \phi(E_{l,m}^{(n)})e_j^{(k)},e_i^{(k)}\rangle
 = \langle E_{l,m}^{(n)} Ve_j^{(k)},Ve_i^{(k)}\rangle
 = v_{m,j} \overline{v_{l,i}}.\label{e:rim}
\end{equation} 
We take into account that the reindexing $r=r(l,i)=(l-1)k+i$ is a one-to-one 
correspondence between $\{r\mid 1\leq r\leq nk\}$ and 
$\{(l,i)\mid 1\leq l\leq n,\ 1\leq i\leq k\}$, hence \eqref{e:rim} is equivalent with
\begin{equation}\label{e:firs} \phi_{r,s}=f_s \overline{f_r},\quad r,s=1,\ldots,nk,
\end{equation} where $f_r=v_{l,i}$ whenever $r=r(l,i)=(l-1)k+i$. This proves that 
the Choi matrix $\Phi$ has rank $1$.

From the particular case proved before, 
since the correspondence between Choi matrices and completely 
positive maps is bijective and affine, it follows that, if $\phi$ has the representation 
\eqref{e:kraus}, $\rank(\Phi)\leq m$. On the other hand, if $m=\rank(\Phi)$, then 
$\Phi=\Phi_1+\Phi_2+\cdots +\Phi_m$ with positive matrices $\Phi_j$, 
$\rank(\Phi_j)=1$ for all $j=1,\ldots,m$. 
Then, from the particular case proved before, for each $j=1,\ldots,m$, it follows that 
$\Phi_j(A)=V_j^*AV_j$ for 
some $V_j\in M_{k,n}$ and all $A\in M_n$, hence $\Phi$ has the Kraus 
representation \eqref{e:kraus}.
\end{proof}

\subsection{$*$-Subspaces.}\label{ss:stars}
 For a fixed natural number $m$, $\cS\subseteq M_m$ 
is called a \emph{$*$-subspace} if it is a linear subspace of $M_m$ that is stable 
under taking adjoints. 
Note that, both the real part and imaginary part of matrices in $\cS$ are 
in $\cS$ and hence $\cS$ is linearly generated by the real 
subspace $\cS^\mathrm{h}$ 
of all its Hermitian matrices. Also, 
$\cS^+=\{A\in \cS\mid A\geq 0\}$ is a cone in $\cS^\mathrm{h}$ but, in general, 
$\cS^+$ may fail to linearly generate $\cS^\mathrm{h}$.
Recall \cite{Pau} 
that a $*$-subspace $\cS$ in $M_m$ is called an \emph{operator system} 
if the identity matrix $I_m\in \cS$.  Any operator system $\cS$ is linearly
generated by $\cS^+$, e.g.\ observing that any Hermitian matrix $B\in\cS$ can be 
written
\begin{equation*} B=\frac{1}{2}(\|B\| I_m+B)-\frac{1}{2}(\|B\|I_m-B),
\end{equation*} hence a difference of two positive semidefinite matrices in $\cS$.
The next proposition provides different characterisations of those $*$-subspaces 
$\cS$ of matrices that are linearly generated by $\cS^+$, as well as a model that 
points out the distinguished role  of operator systems. We need first to recall a 
technical lemma.

\begin{lemma}\label{l:range} Given two matrices $A,B\in M_m^+$, we have 
$B\leq \alpha A$, for some some $\alpha>0$, if and only if 
$\ran(B)\subseteq \ran(A)$.
\end{lemma}

\begin{proof} A folklore result in operator theory, e.g.\ see 
\cite{Douglas}, says that for two matrices $A,B\in M_m$, the inequality 
$BB^*\leq \alpha AA^*$, for some $\alpha>0$, 
is equivalent with $\ran(B)\subseteq \ran(A)$. Consequently, 
if $A,B\in M_m^+$ then $B\leq A$ if and only if 
$\ran(B^{1/2})\subseteq \ran(A^{1/2})$. From here the statement follows since
we have $\ran(P)=\ran(P^{1/2})$ for any positive semidefinite matrix $P$.
\end{proof}

\begin{proposition}\label{p:star} Let $\cS$ be a $*$-space in $M_m$. The following 
assertions are equivalent:
\begin{itemize}
\item[(i)] $\cS$ is linearly generated by $\cS^+$.
\item[(ii)] There exists $A\in\cS^+$ such that for any $B\in\cS^\mathrm{h}$ we have 
$B\leq \alpha A$ for some $\alpha>0$.
\item[(iii)] For any $B\in \cS^\mathrm{h}$ there exists $A\in\cS^+$ with $B\leq A$.
\item[(iv)] There exists $T\in M_m$ a matrix of rank $r$, with 
$\ran(T)=\CC^r\oplus 0\subseteq M_m$, and 
an operator system $\cT\subseteq M_r$ such that 
\begin{equation}\label{e:tes} \cS =T^*(\cT\oplus 0_{m-r})T,
\end{equation} where $0_{m-r}$ denotes the $(m-r)\times (m-r)$ null matrix.
\end{itemize}
\end{proposition}

\begin{proof} (i)$\Ra$(ii). Assuming that $\cS$ is linearly generated by $\cS^+$, 
let $A$ be a matrix in $\cS^+$ of maximal rank. We first show that, for any 
$B\in\cS^+$ we have $B\leq \alpha A$ for some $\alpha>0$. To this end, assume 
that this is not true hence, by Lemma~\ref{l:range}, $\ran(B)\not\subseteq \ran(A)$ 
hence, $\ran(A)$ is a proper subspace of $\ran(A)+\ran(B)$. 
Since $A,B\leq A+B$, again by Lemma~\ref{l:range} it follows
$\ran(B)+\ran(A)\subseteq \ran(A+B)$. But then, $A+B\in\cS^+$ has bigger rank 
than $A$, which contradicts the choice of $A$.

Let now $B\in\cS^\mathrm{h}$ be arbitrary. By assumption, $B=B_1-B_2$ with 
$B_j\in \cS^+$ for $j=1,2$ hence, by what has been proven before, there exist 
$\alpha>0$ such that $B_1\leq \alpha A$, hence $B\leq B_1\leq \alpha A$.

(ii)$\Ra$(iii). This implication is obvious. 

(iii)$\Ra$(i). Since $\cS$ is a $*$-subspace, in order to prove that $\cS$ is linearly 
generated by $\cS^+$, it is sufficient to prove that $\cS^\mathrm{h}$ is (real) linearly 
generated by $\cS^+$. To see this, let $B\in\cS^\mathrm{h}$ be arbitrary. By 
assumption, there exist $A_j\in\cS^+$, $j=1,2$, such that $B\leq A_1$ and 
$-B\leq A_2$ hence, letting $A=A_1+A_2\in\cS^+$, we have
\begin{equation*}B=\frac{1}{2}(A-B)-\frac{1}{2}(A+B),
\end{equation*} where $A-B,A+B\in\cS^+$.

(ii)$\Ra$(iv). Let $A\in\cS^+$ be a matrix 
having the property that for any $B\in\cS^\mathrm{h}$ there exists 
$\alpha>0$ such that $B\leq \alpha A$. 
By Lemma~\ref{l:range}, it follows that for any 
$B\in\cS^+$ we have $\ran(B)\subseteq \ran(A)$ hence, since $\cS$ is linearly 
generated by $\cS^+$, it follows that for any $B\in\cS$ we have 
$\ran(B)\subseteq \ran(A)$, in particular, $\ran(A)$ reduces $B$ and
\begin{equation*}B=\left[\begin{matrix} B_0 & 0 \\ 0 & 0 \end{matrix}\right],\quad 
\mbox{ w. r. t. } \CC^m=\ran(A)\oplus \ker(A).
\end{equation*} Letting $r$ denote the rank of $A$, 
observe now that $A_0$ is positive semidefinite and 
invertible as a linear transformation in $\ran(A)$, hence 
\begin{equation*}\cT_0=\{A_0^{-1/2}B_0A_0^{-1/2}\mid B\in\cS\}\end{equation*}
is an operator system in $\cB(\ran(A))$. Then consider a unitary transformation $V$ 
in $M_m$ such that it maps $\ran(A)$ to $\CC^r$ and $\ker(A)$ to $\CC^{m-r}$, 
letting
\begin{equation*} \cT=V\cT_0V^*\mbox{ and } T=VA^{1/2}
\end{equation*}
the conclusion follows.

(iv)$\Ra$(i). This implication is clear.
\end{proof}

\begin{corollary}\label{c:pd} If the $*$-subspace $\cS$ of $M_m$ 
contains a positive definite matrix, then $\cS$ is linearly generated by $\cS^+$.
\end{corollary}

\begin{proof} Indeed, if $P\in\cS$ is positive definite, then 
$\cT=P^{-1/2}\cS P^{-1/2}$ is an operator system and then $\cS=P^{1/2}\cT P^{1/2}$
is linearly generated by $\cS^+$.
\end{proof}

In the following we will use a particular case of the celebrated theorem 
of L.~Kantorovich \cite{Kantorovich}, see also Theorem~I.30 in 
\cite{AliprantisTourky}, of Hahn-Banach type.

\begin{lemma}\label{l:kant} Let $\cS$ be an $*$-subspace of
$M_m$ that is linearly generated by $\cS^+$, 
and let $f\colon \cS\ra\CC$ be a positive linear map, 
in the sense that it maps any element $A\in\cS^+$ to a nonnegative number 
$f(A)$. Then, there exists a positive linear functional $\widetilde f\colon M_m\ra \CC$ 
that extends $f$. 
\end{lemma}

\begin{proof} Briefly, the idea is to consider the $\RR$-linear functional
$f_\mathrm{h}=f|\cS^\mathrm{h}$ 
and note that $f_\mathrm{h}$ is positive. By Proposition~\ref{p:star}, there exists 
$A\in \cS^+$ such that for all $B\in\cS^\mathrm{h}$ there exists $\alpha>0$ with
$B\leq \alpha A$. By Lemma~\ref{l:range}, we have $\ran(B)\subseteq \ran(A)$ 
for all $B\in\cS^\mathrm{h}$.
Let $p\colon \cB(\ran(A))^\mathrm{h}\ra \RR$ 
be defined by
\begin{equation}\label{e:pe} p(C)=\inf\{ f_\mathrm{h}(B)\mid 
C\leq B\in \cS^\mathrm{h}
\},\quad C\in \cB(\ran(A))^\mathrm{h}.\end{equation}
Then $p$ is a sublinear functional on the $\RR$-linear space 
$\cB(\ran(A))^\mathrm{h}$ and $f(B)=p(B)$ for all $B\in\cS^\mathrm{h}$. 
By the Hahn-Banach Theorem, there exists a linear 
functional $g\colon \cB(\ran(A))^\mathrm{h}\ra \RR$ 
that extends $f_\mathrm{h}$ and such that $g(B)\leq p(B)$ for all 
$B\in \cB(\ran(A))^\mathrm{h}$. Then, for any $B\in M_m^+$, since $-B\leq 0$ it follows
$-g(B)=g(-B)\leq p(-B)\leq f_\mathrm{h}(0)=0$, hence $g(B)\geq 0$. Then, let 
$\widetilde f$ be the canonical extension of $g$ to $\cB(\ran(A))=
\cB(\ran(A))^\mathrm{h}+\iac \cB(\ran(A))^\mathrm{h}$, in the usual way, and finally
extend $\widetilde f$ to $M_m$ by letting $\widetilde f(B)=
\widetilde f(P_{\ran(A)}B|\ran(A))$ for all $B\in M_m$, where $P_{\ran(A)}$ denotes the orthogonal projection of $\CC^m$ onto $\ran(A)$.
\end{proof}

\subsection{The Smith-Ward Functional.}\label{ss:sw}
In the following we recall a technical concept introduced by R.R.~Smith 
and J.D.~Ward, 
cf.\ the proof of Theorem 2.1 in \cite{SmithWard}, and there used to provide
another proof to the Arveson's Hahn-Banach Theorem \cite{Arveson}
for completely positive 
maps, see also Chapter 6 in \cite{Pau}. Consider $\cS$ a 
subspace of $M_n$.
Note that, 
for any $k\in\NN$, $M_k(\cS)$, the collection of all $k\times k$ block-matrices with 
entries in $\cS$, canonically identified with $M_k\otimes \cS$, is
embedded into the $C^*$-algebra $M_k(M_n)=M_k\otimes M_n$ 
and hence it inherits a natural order relation, in particular, positivity of its elements is
well-defined. If $\cS$ is a $*$-subspace then $M_k(\cS)$ is a $*$-subspace as well 
and if, in addition, the $*$-subspace $\cS$ is linearly generated by the cone of its 
positive semidefinite matrices, the same is true for $M_k(\cS)$, e.g.\ by 
Proposition~\ref{p:star}. Thus, a linear map $\phi\colon \cS\ra M_k$, is called 
\emph{positive} if it maps any positive semidefinite matrix from $\cS$ 
to a positive semidefinite matrix in $M_k$. Moreover, for $m\in \NN$, letting 
$\phi_m=I_m\otimes \phi\colon M_m\otimes \cS\ra M_m\otimes M_k$, 
by means of the canonical identification of $M_m\otimes \cS$ with $M_m(\cS)$, 
the $C^*$-algebra of all $m\times m$ block-matrices with entries elements from 
$\cS$, it follows that
\begin{equation*} \phi_m([a_{i,j}]_{i,j=1}^m)=[\phi(a_{i,j})]_{i,j=1}^m,\quad 
[a_{i,j}]_{i,j=1}^m\in M_m(\cS).
\end{equation*}
Then, $\phi$ is called \emph{$m$-positive} if $\phi_m$ is a positive map, and it is 
called \emph{completely positive} 
if it is $m$-positive for all $m\in\NN$. However, positive 
semidefiniteness in the sense of \eqref{e:psd} cannot be defined, at this level of
generality.

To any linear map $\phi\colon\cS\ra M_k$, where $\cS\subseteq M_n$ is some 
linear subspace, one associates a linear functional 
$s_\phi\colon M_{k}(\cS)\ra \CC$, via the canonical identification of 
$M_k(\cS)\simeq M_k\otimes \cS$, by
\begin{align}\label{e:swdir} s_\phi([A_{i,j}]_{i,j=1}^k)
& =\sum_{i,j=1}^k\langle \phi(A_{i,j}) e_j^{(k)},e_i^{(k)}\rangle_{\CC^k}\\
& = \langle (I_k\otimes \phi([A_{i,j}]_{i,j=1}^k))e^{(k)},e^{(k)}\rangle _{\CC^{k^2}}
\nonumber\\
& =\langle [\phi(A_{i,j})]_{i,j=1}^k e^{(k)},e^{(k)}\rangle_{\CC^{k^2}}\nonumber
\end{align} where $[A_{i,j}]_{i,j=1}^k\in M_k(\cS)$, that is, it is a $k\times k$ 
block-matrix, in which each block $A_{i,j}$ is an $n\times n$ matrix from $\cS$,
and $e^{(k)}$ is defined by
\begin{equation}\label{e:e} e^{(k)}=e_1^{(k)}\oplus\cdots\oplus 
e_k^{(k)}\in \CC^{k^2}=\CC^k\oplus \cdots \oplus \CC^k.
\end{equation}
The formula \eqref{e:swdir} establishes a linear isomorphism
\begin{equation}\label{e:swdirect} \cB(\cS,M_k)\ni \phi\mapsto s_\phi\in (M_k\otimes 
\cS)^*\simeq \cB(M_k\otimes \cS,\CC),
\end{equation}
with the inverse transformation
\begin{equation}\label{e:swinverse} 
(M_k\otimes 
\cS)^*\simeq \cB(M_k\otimes \cS,\CC) \ni s\mapsto \phi_s\in\cB(\cS,M_k)
\end{equation}
given by the formula
\begin{equation}\label{e:swinv} \phi_s(A)=[s(E_{i,j}^{(k)}\otimes A)]_{i,j=1}^k,\quad 
A\in \cS.
\end{equation}

The importance of the Smith-Ward functional relies on the facts gathered 
in the following theorem: the 
equivalence of (a) and (d) is a particular case of the Arveson's 
Hahn-Banach Theorem \cite{Arveson}, 
while the equivalence of (a), (b), and (d) is essentially 
due to R.R.~Smith and J.D.~Ward
\cite{SmithWard} as another proof of Arveson's result.

\begin{theorem}\label{t:sw} Let $\cS$ be a $*$-subspace of $M_n$ that is linearly 
generated by $\cS^+$ and let 
$\phi\colon\cS\ra M_k$ be a linear map. The following assertions are equivalent:
\begin{itemize}
\item[(a)] $\phi$ is completely positive.
\item[(b)] $\phi$ is $k$-positive.
\item[(c)] $s_\phi$ is a positive functional.
\item[(d)] There exists $\tilde\phi\in\mathrm{CP}(M_k,M_n)$ that extends $\phi$.
\end{itemize}
\end{theorem}

\begin{proof}
Clearly (a) implies (b), the fact that (b) implies (c) follows from the definition of 
$s_\phi$ as
in \eqref{e:swdir}, while (d) implies (a) is clear as well. The only nontrivial part is 
(c) implies (d). Briefly, following the proofs of Theorem~6.1 
and Theorem~6.2 in \cite{Pau}, the idea is to use Kantorovich's Theorem as in 
Lemma~\ref{l:kant} in order to extend 
$s_\phi$ to a positive functional $\widetilde s$ on $M_k\otimes M_n\simeq
M_k(M_n)\simeq M_{kn}$ then, in view of \eqref{e:swinv}, let $\widetilde 
\phi\colon M_{k}\ra M_n$ be defined by
\begin{equation}\label{e:swinvdoi} \widetilde\phi(A)=[\widetilde s(E_{i,j}^{(k)}
\otimes A)]_{i,j=1}^k,\quad A\in M_n,
\end{equation} and note that $\widetilde\phi$ extends $\phi$. Finally, in order to 
prove that $\widetilde\phi$ is completely positive it is sufficient to prove that it is 
positive semidefinite, see \eqref{e:psd}. To see this, let $m\in\NN$, $A_1,\ldots,A_m
\in M_n$, and $h_1,\ldots,h_m\in \CC^k$ be arbitrary. Then, letting
\begin{equation*} h_j=\sum_{l=1}^k \lambda_{k,l}e_l^{(k)},\quad j=1,\ldots,m,
\end{equation*}
we have
\begin{align*}
\sum_{i,j=1}^m \langle \widetilde\phi(A_i^*A_j)h_j,h_i\rangle_{\CC^k} 
& = \sum_{i,j=1}^m\sum_{l,p=1}^k \lambda_{j,l}\overline{\lambda}_{i,p}\langle
\widetilde\phi(A_i^*A_j)e_l^{(k)},e_p^{(k)}\rangle_{\CC^k} \\
& = \sum_{i,j=1}^m\sum_{l,p=1}^k \lambda_{j,l}\overline{\lambda}_{i,p}\widetilde 
s(A_i^*A_j\otimes E_{p,l}^{(k)}) \\
\intertext{then, for each $i=1,\ldots,m$, letting $B_i$ denote the $k\times k$ matrix 
whose first row is $\lambda_{i,1},\ldots,\lambda_{i,k}$ and all the others are $0$, 
hence
$B_i^*B_j=\sum_{l,p=1}^k \lambda_{j,l}\overline{\lambda}_{i,p} E_{p,l}^{(k)}$,
we have
}
& = \sum_{i,j=1}^m \widetilde s(A_i^*A_j\otimes B_i^*B_j) \\
& = \widetilde s((\sum_{i=1}^m A_i\otimes B_i)^* (\sum_{j=1}^m A_j\otimes B_j)) 
\geq 0.\qedhere
\end{align*}
\end{proof}

Actually, once the facts described in Theorem~\ref{t:sw} are settled, 
it is easy to observe that \emph{ \eqref{e:swdir} and 
\eqref{e:swinv} establish an affine and order preserving bijection 
between the cone $\mathrm{CP}(\cS,M_k)$ and the cone 
$\{ s\colon M_k(\cS)\ra\CC\mid  s\mbox{ linear and positive}\}$}.

\subsection{The Density Matrix.}\label{ss:dm}
We consider $M_{m}$ as a Hilbert space with the Hilbert-Schmidt 
inner product, that is, $\langle C,D\rangle_{\mathrm{HS}}=\tr(D^*C)$, for all 
$C,D\in M_{m}$. Thus, to any linear functional
$s\colon M_{m}\ra\CC$,
by the Representation Theorem for
(bounded) linear functionals on a Hilbert space, in our case $M_{m}$ with the 
Hilbert-Schmidt inner product, one associates uniquely a matrix $D_s\in M_{m}$, 
such that
\begin{equation}\label{e:density}s(C)=\tr(D_s^*C),\quad C\in M_{m}.
\end{equation}
Clearly, $s\mapsto D_s$ is a conjugate 
linear bijection between the dual space of $M_{m}$ and $M_{m}$. 

\begin{remark}\label{r:density}
Using the properties of the trace, it follows that $s$ is 
a positive functional if and only if the matrix $D_s$ is positive semidefinite. Indeed,
if $D_s$ is positive semidefinite, then for all positive semidefinite matrix in 
$M_{m}$, we have $\tr(D_sC)=\tr(C^{1/2}D_sC^{1/2})\geq 0$. 
Conversely, if $\tr(D_sC)\geq 0$ for all positive semidefinite $m\times m$ matrix 
$C$, then for any vector $v$ of length $m$ we have $0\leq \tr (D_s vv^*)=
\tr(v^*D_sv)=v^*D_sv$, hence $D_s$ is positive semidefinite.\end{remark}

From the previous remark, if $s$ is a \emph{state} on $M_{m}$, 
that is, a unital positive linear functional on $M_{m}$, then $D_s$ 
becomes a \emph{density matrix}, that is, a positive semidefinite matrix of trace one.
Slightly abusing this fact, we call $D_s$ the \emph{density matrix} 
associated to $s$, in general.

On the other hand, since the correspondence between linear maps 
$\phi\colon M_n\ra M_k$ and Choi matrices $\Phi_\phi\in M_{kn}$ is a linear 
isomorphism and, via the Smith-Ward linear functional $s_\phi$, the correspondence 
between $\phi$ and the density matrix $D_{s_\phi}$ is a conjugate linear 
isomorphism, it is natural to ask for an explicit relation between the Choi matrix 
$\Phi_\phi$ and $D_{s_\phi}$. In order to do this, we 
first recall the definition of the \emph{canonical shuffle} operators. Briefly, 
this comes
from the two canonical identifications of $\CC^n\otimes \CC^k$ with $\CC^{kn}$, 
more precisely, for each $l\in \{1,\ldots,n\}$ and each $i\in \{1,\ldots,k\}$, we let
\begin{equation}\label{e:shuffle} Ue_{(i-1)n+l}^{(kn)}=e_{(l-1)k+i}^{(kn)}.
\end{equation}
It is clear that $U$ is a unitary operator $\CC^{kn}\ra \CC^{kn}$, hence an 
orthogonal $kn\times kn$ matrix. Also, for a matrix $X$, out of the adjoint 
matrix $X^*$, we consider its transpose $X^T$ as well as its entrywise complex 
conjugate $\overline{X}$.

\begin{proposition}\label{p:trace} For any linear map $\phi\colon M_n\ra M_k$ and 
letting $\Phi$ denote its Choi matrix, cf.\ \eqref{e:choi}, the 
density matrix $D$ associated to the Smith-Ward linear functional 
$s_\phi$, cf.\ \eqref{e:density} and \eqref{e:swdir}, is
\begin{equation}\label{e:jam} D=U^* \overline\Phi U,
\end{equation} where $U$ is the canonical shuffle unitary operator defined at 
\eqref{e:shuffle}.
\end{proposition}

\begin{proof} 
We first note that $\{E^{(k)}_{i,j}\otimes E^{(n)}_{l,m}\mid 
i,j=1,\ldots,k,\ l,m=1,\ldots,n\}$ is an orthonormal basis of $M_k\otimes 
M_n$ with respect to the Hilbert-Schmidt inner product, and that, with respect to the 
canonical identification of $M_k\otimes M_n\simeq M_k(M_n)$, that is, when viewed
as block $k\times k$ matrices with each entry an $n\times n$ matrix, 
with $M_{kn}$, we have
\begin{equation*} E^{(k)}_{i,j}\otimes E^{(n)}_{l,m}=E^{(kn)}_{(i-1)n+l,(j-1)n+m},\quad 
i,j=1,\ldots,k,\ l,m=1,\ldots,n.
\end{equation*} 
Fix $i,j\in\{1,\ldots,k\}$ and $l,m\in\{1,\ldots,n\}$. Then,
\begin{align} s_\phi(E^{(k)}_{i,j}\otimes E^{(n)}_{l,m}) 
  &= s_\phi(E^{(kn)}_{(i-1)n+l,(j-1)n+m}) \nonumber\\
&= \tr(D^* E^{(kn)}_{(i-1)n+l,(j-1)n+m})=\overline d_{(i-1)n+l,(j-1)n+m},\label{e:sefe}
\end{align} where $D=[d_{r,s}]_{r,s=1}^{kn}$ is the matrix representation of $D$, 
more precisely,
\begin{equation*} D=\sum_{r,s=1}^{kn} d_{r,s} E^{(kn)}_{r,s}.
\end{equation*} 
On the other hand, from \eqref{e:prima} we have
\begin{equation}\label{e:sefedoi} s_\phi(E_{i,j}^{(k)}\otimes E_{l,m}^{(n)}) 
=\langle \phi(E_{l,m}^{(n)}) e_j^{(k)},e_i^{(k)}\rangle_{\CC^k}=\phi_{(l-1)k+i,(m-1)k+j}.
\end{equation}
Therefore, in view of \eqref{e:sefe}, \eqref{e:sefedoi} we have
\begin{equation*}  \overline d_{(i-1)n+l,(j-1)n+m}
= s_\phi(E^{(k)}_{i,j}\otimes E^{(n)}_{l,m}) = \phi_{(l-1)k+i,(m-1)k+j}, 
\end{equation*}
hence, taking into account of the definition of the canonical shuffle operator $U$ as
in \eqref{e:shuffle}, the 
equality in \eqref{e:jam} follows. 
\end{proof}

We now come back to the general case of a $*$-subspace $\cS$ in $M_m$. By
analogy with the particular case of the operator system of full matrix algebra $M_m$ 
described before, with respect to the Hilbert-Schmitd inner product 
on $M_{m}$, hence on its subspace $\cS$, to any
linear functional $s\colon \cS\ra \CC$ one uniquely associates an 
$m\times m$ matrix $D_s\in \cS\subseteq M_{m}$ such that 
\begin{equation}\label{e:densityo}
s(C)=\tr(D^*_sC),\quad C\in \cS,
\end{equation} 
and we continue to call $D_s$ the \emph{density matrix} associated to $s$. Clearly, 
this establishes a conjugate linear isomorphism between 
the dual space of $\cS$ and $\cS$. In view of Theorem~\ref{t:sw}, we 
may ask whether the positivity of the linear functional $s$ is equivalent with 
the positive semidefiniteness of its density matrix, as in the case of the full matrix 
algebra $M_n$. Clearly, if the density matrix $D_s$ is positive semidefinite then $s$ 
is a positive linear functional but,
as the following remarks and examples show, the converse may or may not hold.

\begin{remarks}\label{r:hermitian} 
(1) If the $\cS$ is a $*$-subspace of $M_n$ and the 
linear functional $s\colon \cS\ra \CC$ is Hermitian, that is, 
$s(C^*)=\overline{s(C)}$ for all $C\in\cS$, then its density matrix $D$ is Hermitian.  
Indeed, for any $C\in\cS$ we have
\begin{equation*} s(C)=\overline{s(C^*)}=\overline{\tr(D^*C^*)}=\tr(CD)=\tr((D^*)^*C),
\end{equation*} hence, $D^*$ is also a density matrix for $s$. Since the density 
matrix is unique, it follows that $D=D^*$.

(2) If $\cS$ is a $C^*$-subalgebra of $M_m$, not necessarily unital, then for any
positive functional $s\colon \cS\ra \CC$, its density matrix $D$ is positive 
semidefinte. Indeed, in this case $D=D_+-D_-$ with $D_\pm\in\cS^+$ and 
$D_+D_-=0$ hence $0\leq s(D_-)=\tr(DD_-)=-\tr(D_-^2)$ hence $D_-=0$ and 
consequently $D\in\cS^+$.
\end{remarks}

\begin{examples}\label{ex:pp}
(1) We consider the following operator system $\cS$ in $M_3$
\begin{equation}\label{e:os1}
\cS=\bigl\{C=\left[\begin{matrix} a & 0 & b \\ 0 & a & 0 \\ c & 0 & d\end{matrix}\right]
\mid a,b,c,d\in\CC\bigl\},
\end{equation} and note that $\cS^+$ consists on those matrices $C$ as in 
\eqref{e:os1} with $c=\overline{b}$, $a,d\geq 0$, and $|b|^2\leq ad$. Let
\begin{equation*}D=\left[\begin{matrix} 1 & 0 & \sqrt{2} \\ 0 & 1 & 0 \\ \sqrt{2} & 0 & 1 
\end{matrix}\right],
\end{equation*} and note that $D\in\cS$ is Hermitian but it is not positive 
semidefinite: more precisely,
its eigenvalues are $1-\sqrt{2}$, $1$, and $1+\sqrt{2}$. 
On the other hand, for any $C\in\cS^+$, that is, with the notation as in \eqref{e:os1}, $c=\overline{b}$, $a,d\geq 0$, and $|b|^2\leq ad$, we have
\begin{align*} \tr(DC) & =a+\sqrt{2}\, \overline{b}+a +\sqrt{2}\, b +d 
=2a+d+2\sqrt{2}\Re b\\
& \geq 2a+d-2\sqrt{2}|b|\geq 2a+d-2\sqrt{2}\sqrt{ad}
=(\sqrt{2a}-\sqrt{d})^2\geq 0,
\end{align*} hence the linear functional $\cS\ni C\mapsto \tr(DC)\in\CC$ 
is positive.

(2) In $M_2$ we consider the Pauli matrices
\begin{equation}\label{e:pauli} \sigma_0=\left[\begin{matrix} 1 & 0 \\ 0 & 1 \end{matrix}\right],\quad \sigma_1=\left[\begin{matrix} 0 & -\iac \\ \iac & 0 \end{matrix}\right],\quad
\sigma_2=\left[\begin{matrix} 0 & 1 \\ 1 & 0 \end{matrix}\right],\quad
\sigma_3=\left[\begin{matrix} 1 & 0 \\ 0 & -1 \end{matrix}\right],
\end{equation}
that makes an orthogonal basis of $M_2$ with respect to the Hilbert-Schmidt 
inner product. We consider $\cS$ the linear span of $\sigma_0$, $\sigma_1$, and 
$\sigma_2$, more precisely,
\begin{equation}\label{e:pa3} \cS=\{ C=\left[\begin{matrix} \alpha & \beta \\ 
\gamma & \alpha\end{matrix}\right] \mid \alpha,\beta,\gamma\in\CC\}.
\end{equation} Note that $\cS$ is an operator system but not an algebra. However, 
we show that, an arbitrary matrix $D\in\cS$ is positive semidefinite if and 
only if $\tr(D^*C)\geq 0$ for all $C\in\cS^+$.

To this end, note that a matrix $C$ as in \eqref{e:pa3} is positive semidefinite if 
and only if $\gamma=\overline{\beta}$, $\alpha\geq 0$, and $|\beta|^2\leq \alpha^2$. 
Let $D\in\cS$, that is,
\begin{equation*} D=\left[\begin{matrix} a & b \\ 
c & a\end{matrix}\right], 
\end{equation*}  such that $\tr(D^*C)\geq 0$ for all $C\in\cS^+$. From 
Remark~\ref{r:hermitian} it follows that $D$ is Hermitian, hence $a$ is real and 
$c=\overline{b}$, and the condition $\tr(D^*C)\geq 0$ can be equivalently written as
\begin{equation}\label{e:cond} a\alpha+\Re(\overline\beta b)\geq 0\mbox{ whenever } 
\alpha\geq 0\mbox{ and } |\beta|^2\leq \alpha^2.
\end{equation} Letting $\beta=0$ implies that $a\geq 0$. We prove that $|b|^2\geq 
a^2$. If $a=0$ then from \eqref{e:cond} it follows that $b=0$. If $a>0$ and 
$|b|^2>a^2$ then letting $\alpha=a$ and $\beta=-a|b|/\overline{b}$, 
we obtain $0\leq a\alpha+
\Re(\overline \beta b)=a^2-a|b|=a(a-|b|)<0$, a contradiction. Hence $|b|^2\geq 
a^2$ must hold, and we have proven that $D$ is positive semidefinite.
\end{examples}

The concept of density matrices associated to linear functionals on $*$-subspaces 
opens the possibility of generalising the concept of a Choi matrix for linear 
maps with domains $*$-subspaces. Note that the definition of the Choi matrix, see 
\eqref{e:choi}, involves essentially the matrix units which, generally, 
are not available in 
operator systems. However, in view of Proposition~\ref{p:trace} we can proceed as 
follows. 
First consider the Smith-Ward functional $s_\phi$ defined as in \eqref{e:swdir},
then consider the density matrix $D_{\phi}$ associated to $s_\phi$ as in 
\eqref{e:densityo}, and finally define $\Phi_\phi$ by
\begin{equation} \label{e:choios} \Phi_\phi =U\,\overline{D_\phi}\, U^*,
\end{equation} where the bar denotes the entrywise complex 
conjugation and $U$ denotes 
the canonical shuffle unitary operator as in \eqref{e:shuffle}. Clearly $\Phi_\phi$ is an 
$kn\times kn$ matrix and, in case $\phi\in\mathrm{CP}(\cS,M_k)$, the Choi matrix 
$C_\phi$ defined as in \eqref{e:choios} is Hermitian but, at this level of generality, 
depending on the $*$-subspace $\cS$, its 
positive definiteness is not guaranteed. However, if the Choi matrix 
$\Phi_\phi$ is positive semidefinite, then $\phi\in \mathrm{CP}(\cS,M_k)$.

\subsection{Operator Systems Generated by Matrix Units.}\label{ss:osgmu}
For a fixed natural  number $m$ let $\cS$ be an operator system in $M_m$. We are 
interested by the special case when $\cS$ is linearly generated by a subset of matrix 
units in $M_m$, that is, there exists a subset 
$S\subseteq \{1,\ldots,m\}^2$ such that $\cS=\lin\{E^{(m)}_s 
\mid s\in S\}$. 

\begin{remarks}\label{r:mut} In the following we use the interpretation of subsets 
$S\subseteq \{1,\ldots,m\}^2$ as relations on the set 
$\{1,\ldots,m\}$. Let $S$ be a relation on $\{1,\ldots,m\}$ and let 
$\cS=\lin\{E^{(m)}_s \mid s\in S\}$ be the linear subspace in $M_m$ generated by 
the matrix unit indexed in $S$. 

\nr{1} The linear space $\cS$ is an operator system in $M_m$ if and only if $S$ 
is reflexive and symmetric.

\nr{2} The linear space $\cS$ is a unital $*$-subalgebra of $M_m$ if and only if $S$
is an equivalence relation.
\end{remarks}

\begin{theorem}\label{t:mut} 
Let $\cS$ be an operator system in $M_m$ linearly generated by matrix units. 
The following assertions are equivalent:
\begin{itemize}
\item[(a)] Any positive linear functional $s\colon\cS\ra \CC$ has a positive 
semidefinite density matrix.
\item[(b)] $\cS$ is an algebra (and hence a unital $*$-subalgebra of $M_m$).
\end{itemize}
\end{theorem}

\begin{proof} (a)$\Ra$(b). We divide the proof in three steps:

\emph{Step 1.}
First observe that for $m=1$ or $m=2$ any 
operator system $\cS\subseteq M_m$ generated by a set of matrix units
is an algebra hence, in view of  Remark~\ref{r:hermitian}.(2), there is nothing to 
prove.

\emph{Step 2.}
We consider $m=3$ so let $\cS$ be an operator system in $M_3$ that is 
linearly generated by a reflexive and symmetric relation $S\subseteq\{1,2,3\}^2$.
Then $S$ necessarily contains all the diagonal $S_\mathrm{d}=\{(1,1),(2,2),(3,3)\}$. 
On the other hand, due to the symmetry 
condition on $S$, it may contain only $3$, $5$, $7$, or $9$ elements. If $S$ 
contains either $3$, $5$, or $9$ elements it is easy to see that $\cS$ is an algebra. 
Thus, we are left only with the investigation of the case when $S$ has exactly $7$ 
elements, and these are the cases when 
$S=S_\mathrm{d}\cup\{(1,2),(2,1),(1,3),(3,1)\}$, 
$S=S_\mathrm{d}\cup\{(1,2),(2,1),(3,2),(2,3)\}$, 
$S=S_\mathrm{d}\cup\{(2,3),(3,2),(1,3),(3,1)\}$. 
Note that these three cases correspond to a circular permutation of one of them 
and hence the proof for any one of these would be sufficient.

Let $S=\{(1,1),(2,2),(3,3),(1,2),(2,1),(1,3),(3,1)\}$. In the following we prove that the 
corresponding operator system $\cS=\lin S$, that is not an algebra, has at least one 
(actually we prove that there are infinitely many) positive linear functional 
whose density matrix is not positive semidefinite. To see this, first note that
\begin{equation}\label{e:cese} \cS=\{ C=
\left[\begin{matrix} a & b & c \\ f & d & 0 \\ g & 0 & e\end{matrix}\right] \mid 
a,b,c,d,e,f,g \in \CC\},
\end{equation} and that $\cS^+$ is the collection of all matrices $C$ as in 
\eqref{e:cese} subject to the following conditions
\begin{equation}\label{e:condes} a\geq 0,\ d\geq 0,\ e\geq 0, f=\overline{b}, 
g=\overline{c},\ |b|^2\leq ad,\ |c|^2\leq ae,\ |b|^2e+|c|^2 d\leq ade. 
\end{equation}
For each $1/\sqrt{2}<\rho \leq 1$ consider the Hermitian matrix
\begin{equation}\label{e:dero} D_\rho = \left[\begin{matrix} 1 & \rho & \rho \\ 
\rho & 1 & 0 \\ \rho & 0 & 1
\end{matrix}\right].
\end{equation} It is easy to see that $D_\rho$ is indefinite for each 
$1/\sqrt{2}<\rho \leq 1$.

We prove that the corresponding functional $s_\rho=\tr(D_\rho \cdot)$ is positive 
for each $1/\sqrt{2}<\rho \leq 1$. To see this, let $C\in\cS^+$ be arbitrary, that is,
with the notation as in \eqref{e:cese}, the conditions \eqref{e:condes} must hold. 
Then
\begin{align*} s_\rho(C) & = \tr(D_\rho C)=a+d+e+2\rho\Re(b+c) \\
& \geq a+d+e-2\rho(|b|+|c|) \\
\intertext{and, taking into account that $|b|+|c|$ has its maximal value 
$\sqrt{a(e+d)}$ when the constraints \eqref{e:condes} hold, it follows that}
& \geq a+d+e-2\rho\sqrt{a(e+d)} \\
& \geq a+d+e-2\sqrt{a(e+d)}=(\sqrt{a}-\sqrt{e+d})^2\geq 0.
\end{align*}

\emph{Step 3.} Assume now that $m>3$ and assume that $\cS$ is an 
operator system in $M_m$ that is not an algebra. By Remark~\ref{r:mut}.(2), there
exist distinct $i,j,l\in\{1,\ldots,m\}$ such that $(i,j),(j,l)\in S$ but $(i,l)\not\in S$. 
Modulo a reindexing, without loss of generality we can assume that $i=2$, $j=1$, 
and $l=3$. For each $1/\sqrt{2}<\rho \leq 1$ we consider the matrix $D_\rho$ 
as in \eqref{e:dero} and let $\widetilde D_\rho=D_\rho\oplus 0\in M_m$. From what 
has been proven in Step 3 it follows that the functional 
$s_\rho=\tr(\widetilde D_\rho\cdot)$ on $M_m$ is positive but its density matrix 
$\widetilde D_\rho$ is not positive semidefinite.

(b)$\Ra$(a). This is a consequence of Remark~\ref{r:hermitian}.(2).
\end{proof}

\section{Main Results}\label{s:mr}

\subsection{The General Case.}\label{ss:gc}
Let now $A_1,\ldots,A_N\in M_n$ and $B_1,\ldots,B_N\in M_k$ be the given 
interpolation data with respect to the Interpolation Problem, see the Introduction. 
We recall the notation
$\bA=(A_1,\ldots,A_N)$, called the \emph{input data} and, similarly, 
$\bB=(B_1,\ldots,B_N)$, called the \emph{output data}.
Since we are looking for $\phi\in\mathrm{CP}(M_n,M_k)$, hence for Hermitian 
maps $\phi\colon M_n\ra M_k$, such that the interpolation condition holds
\begin{equation}\label{e:princa}\varphi(A_\nu) =B_\nu,\mbox{ \it for all }\nu=1,\ldots,N,
\end{equation}  without loss of generality we can assume 
that all $A_\nu$ and all $B_\nu$ are Hermitian, otherwise we may
increase the number of the data by splitting each entry into 
its real and its imaginary parts, respectively. 
Also, without loss of generality, 
we can assume that $A_1,\ldots,A_N$ are linearly independent, 
otherwise some linearly dependence consistency conditions on $B_1,\ldots,B_N$ 
should be imposed. 
On the other hand, since the required maps $\phi$ should be linear, the constraint
\eqref{e:princa} actually determines $\phi$ on the space 
\begin{equation}\label{e:sa} \cS_\bA=\lin\{A_1,\ldots,A_N\},
\end{equation}
which is a $*$-subspace due to the fact that 
all $A_\nu$ are Hermitian matrices. In conclusion, without loss of generality, 
we work under the following hypotheses on the data: 
\begin{itemize}
\item[(a1)] \emph{All matrices $A_1,\ldots,A_N\in M_n$ and 
$B_1,\ldots,B_N\in M_k$ are Hermitian.}
\item[(a2)] \emph{The set of matrices $\{A_1,\ldots,A_N\}$ is linearly independent.}
\end{itemize}
Thus, $\cS_\bA$ is 
$*$-subspace of $M_n$ for which $A_1,\ldots, A_N$ is a linear basis. Having in 
mind the approach of the Interpolation Problema
through the Arveson's Hahn-Banach Theorem and Smith-Ward linear functional,
$\cS_\bA$ might be required to be linearly generated by $\cS_\bA^+$. Thus, we will 
also consider special cases when, in addition to the hypotheses (a1) and (a2), the 
following condition might be imposed on the data:
\begin{itemize}
\item[(a3)] \emph{$\cS_\bA$ is linearly generated by $\cS_\bA^+$.}
\end{itemize}

\begin{remark}\label{r:compact}
Recalling the definition of $\cC_{\bA,\bB}$ as in 
\eqref{e:data}, the set of solutions of the Interpolation Problem, 
observe that $\cC_{\bA,\bB}$ is convex and closed. If
$\cS_\bA$ contains a positive definite matrix of rank $n$, in particular, if $\cS_\bA$ 
is an operator system, then $\cC_{\bA,\bB}$ is 
bounded as well, hence compact. Indeed, if $\cS_\bA$ is an operator system,
this follows from the fact, 
e.g.\ see Proposition~3.6 in \cite{Pau}, that 
$\|\phi\|= \|\phi(I_n)\|$ and, since $I_n\in \cS_\bA$, the positive semidefinite matrix  
$\phi(I_n)$ is fixed and independent of 
$\phi\in\cC_{\bA,\bB}$.  The general case follows now by Proposition~\ref{p:star}.
However, the same Proposition~\ref{p:star} shows that 
assuming that $\cS_\bA$ is generated by $\cS_\bA^+$ is not sufficient for
the compactness of $\cC_{\bA,\bB}$. \end{remark}

In order to approach the Interpolation Problem, it is natural to associate a linear map 
$\phi_{\bA,\bB}\colon \cS_\bA\ra M_k$ to the data $\bA$ and $\bB$ by letting
\begin{equation}\label{e:fia} \phi_{\bA,\bB}(A_\nu)=B_\nu,\quad \nu=1,\ldots,N,
\end{equation} and then uniquely extending it by linearity to the whole 
$*$-subspace $\cS_\bA$.
Then, having in mind the Smith-Ward linear functional \eqref{e:swdir}, let
\begin{equation}\label{e:es} s_{\bA,\bB}(E_{i,j}^{(k)}\otimes A_\nu)
=\langle B_\nu e_j^{(k)},e_i^{(k)}\rangle_{\CC^k}=b_{i,j,\nu},
\quad i,j=1,\ldots,N,\ \nu=1,\ldots,N,
\end{equation}
where
\begin{equation}\label{e:be} B_\nu=\sum_{i,j=1}^k b_{i,j,\nu} E_{i,j}^{(k)},
\quad \nu=1,\ldots,N.
\end{equation}
Since $\{E_{i,j}^{(k)}\otimes A_\nu\mid i,j=1,\ldots k,\ 
\nu=1,\ldots,N\}$ is a basis for $M_k\otimes \cS_\bA$, it follows that 
$s_{\bA,\bB}$ admits a unique extension to a linear functional $s_{\bA,\bB}$ 
on $M_k(\cS_\bA)$. Note that, with respect to 
the transformations \eqref{e:swdir} and 
\eqref{e:swinverse}, the functional $s_{\bA,\bB}$ corresponds to the map
$\phi_{\bA,\bB}$, and vice-versa.

To the linear functional $s_{\bA,\bB}$ one
also uniquely associates its density matrix $D_{\bA,\bB}$ as in \eqref{e:densityo}, 
more precisely,
\begin{equation}\label{e:dab} s_{\bA,\bB}(C)=\tr(D_{\bA,\bB}^* C),\quad C\in\cS_\bA,
\end{equation}
that can be explicitly calculated in terms of input-output data $\bA$ and $\bB$, 
as follows.

\begin{proposition}\label{p:dens} 
Let the data $A_1,\ldots,A_N$ and $B_1,\ldots,B_N$ 
satisfy the assumptions \emph{(a1)} and \emph{(a2)}. 
Then, the density matrix $D_{\bA,\bB}$ of the linear functional $s_{\bA,\bB}$ is
\begin{equation}\label{e:de} 
D_{\bA,\bB}=\sum_{\nu=1}^N \sum_{i,j=1}^k d_{i,j,\nu} E_{i,j}^{(k)}\otimes A_\nu,
\end{equation}
where, for each pair $i,j=1,\ldots,k$, the numbers $d_{i,j,1},\ldots,d_{i,j,N}$ are the 
unique solutions of the linear system
\begin{equation}\label{e:dem}
 \sum_{\mu=1}^N d_{i,j,\mu} \tr(A_\mu A_\nu)=\overline{b}_{i,j,\nu},
 \quad \nu=1,\ldots,N,
\end{equation} and the numbers $b_{i,j,\nu}$ are defined at \eqref{e:be}.
\end{proposition}

\begin{proof} 
Clearly, the density matrix $D_{\bA,\bB}$ can be represented in terms of the 
basis 
$\{E^{(k)}_{i,j}\otimes A_\nu\mid i,j=1,\ldots,k,\ \nu=1,\ldots,N\}$ as in
\eqref{e:de}, so we only have to show that \eqref{e:dem} holds. To this end, note that
\begin{equation*} D_{\bA,\bB}^*=\sum_{i,j=1}^k \sum_{\nu=1}^N \overline{d}_{i,j,\nu} 
E^{(k)}_{j,i} \otimes  A_\nu,
\end{equation*} recalling that $A_\nu$ are Hermitian matrices, by assumption.
Then, in view of \eqref{e:es}, for each $i,j=1,\ldots,k$ and each $\nu=1,\ldots,N$, 
we have
\begin{align*} b_{i,j,\nu} & = s_{\bA,\bB}(E^{(k)}_{i,j}\otimes A_\nu)
=\tr(D_{\bA,\bB}^*(E^{(k)}_{i,j}\otimes A_\nu)) \\
& = \tr(\sum_{i',j'=1}^k\sum_{\mu=1}^N \overline{d}_{i',j',\mu}(E^{(k)}_{j',i'}E^{(k)}_{i,j}
\otimes A_\mu A_\nu)) \\
\intertext{then, taking into account that $E^{(k)}_{j',i'}E^{(k)}_{i,j}
=\delta_{i',i}E^{(k)}_{j',j}$, we have}
& = \sum_{\mu=1}^N\sum_{j'=1}^k \overline{d}_{i,j',\mu}\tr(E^{(k)}_{j',j})
\tr(A_\mu A_\nu)\\
\intertext{and, since $\tr(E^{(k)}_{j',j})=\delta_{j',j}$, we have}
& = \sum_{\mu=1}^N\overline{d}_{i,j,\mu}\tr(A_\mu A_\nu).
\end{align*}
Finally, observe that the matrix $[\tr(A_\mu A_\nu)]_{\mu,\nu=1}^N$ is the Gramian 
matrix of the linearly independent system $A_1,\ldots,A_N$ with respect to the 
Hilbert-Schmidt inner product, hence positive definite and, in particular, nonsingular. 
Therefore, the system \eqref{e:dem} has unique solution.
\end{proof}

\begin{lemma}\label{l:orta} Let $\cS$ be a $*$-subspace in $M_n$ and 
let $\cS^\perp$ be the orthogonal
complement space associated to $\cS$ with respect to the Hilbert-Schmidt inner 
product
\begin{equation}\label{e:ort} \cS^\perp=\{E\in M_n\mid \tr(C^*E)=0,\mbox{ for all }
C\in\cS\}.\end{equation}
Then:

\nr{a} $\cS^\perp$ is a $*$-subspace of $M_n$, 
hence linearly generated by its Hermitian matrices.

\nr{b} $(M_k\otimes \cS)^\perp=M_k\otimes \cS^\perp$, in particular, 
$(M_k\otimes\cS)^\perp$ is a $*$-subspace of $M_k\otimes M_n$.

\nr{c} If $\cS$ is an operator system then 
any matrix $C\in \cS^\perp$ has zero trace, in particular 
$\cS^\perp\cap M_n^+=\{0\}$.

\nr{d}  If $\cS$ is an operator system then any matrix in
$(M_k\otimes\cS)^\perp$ has zero trace, hence $(M_k\otimes\cS)^\perp$
does not contain nontrivial positive semidefinite matrices.
\end{lemma}

\begin{proof} (a) Clearly, $\cS^\perp$ is a subspace of $M_n$. Let $E\in\cS^\perp$, 
hence $\tr(E^*C)=0$ for all $C\in\cS$. Then, $0=\overline{\tr(E^*C)}=\tr(C^*E)
=\tr(EC^*)=\tr((E^*)^*C^*)$ for all $C\in\cS$ and, since $\cS$ is stable under taking 
adjoints, this implies that $E^*\in\cS^\perp$.

(b) A moment of thought shows that 
$M_k\otimes\cS^\perp\subseteq (M_k\otimes\cS)^\perp$. On the other hand, 
$\dim((M_k\otimes\cS)^\perp)=k^2n^2-k^2\dim(\cS)=k^2(n^2-\dim(\cS))
=\dim(M_k\otimes\cS^\perp)$, hence the desired conclusion follows.

(c) This is a consequence of the fact that $I_n\in\cS$ and the fact that the trace 
is faithful.

(d) This is a consequence of the statements (b) and (c). 
\end{proof}

\begin{theorem}\label{t:gen} Let the data $A_1,\ldots,A_N\in M_n$ and 
$B_1,\ldots,B_N\in M_k$ be given and subject to the assumptions \emph{(a1)} 
and \emph{(a2)}, let $\phi_{\bA,\bB}$ be the linear map defined at \eqref{e:fia},
let $s_{\bA,\bB}$ be the linear functional defined at \eqref{e:es} and the density 
matrix $D_{\bA,\bB}$ associated to $s_{\bA,\bB}$ as in \eqref{e:densityo}. 
Also, let $\cS_\bA^\perp$ be the 
orthogonal complement space associated to $\cS_\bA$ with respect to the 
Hilbert-Schmidt inner product, see \eqref{e:ort}.

The following assertions are equivalent:

\nr{i} There exists $\phi\in\mathrm{CP}(M_n\, M_k)$ such that 
$\phi(A_\nu)=B_\nu$ for all $\nu=1,\ldots,N$.

\nr{ii} The affine space $D_{\bA,\bB}+M_k\otimes\cS_\bA^\perp$ contains at least 
one positive semidefinite matrix.
\end{theorem}

\begin{proof} (i)$\Ra$(ii). Let $\phi\in\mathrm{CP}(M_n\, M_k)$ be such that 
$\phi(A_\nu)=B_\nu$ for all $\nu=1,\ldots,N$, hence $\phi$ extends the linear map 
$\phi_{\bA,\bB}$, and let 
$s_\phi\colon M_k(M_n)\ra\CC$ be the Smith-Ward linear functional associated to 
$\phi$ as in \eqref{e:swdir}. Since $\phi$ is completely positive, it follows that 
$s_\phi$ is positive. Further, let $D_\phi\in M_{kn}$ be the density matrix of 
$s_\phi$, cf.\ \eqref{e:density}, hence, by Remark~\ref{r:density}, $D_\phi$ is 
positive semidefinite. On the other hand, since $\phi$ extends $\phi_{\bA,\bB}$, 
it follows that $s_\phi$ extends $s_{\bA,\bB}$, hence $D_\phi=D_{\bA,\bB}+P$ 
for some $P\in (M_k\otimes \cS_\bA)^\perp=M_k\otimes \cS_\bA^\perp$.

(ii)$\Ra$(i). Let $D=D_{\bA,\bB}+P$ be positive semidefinite, for some 
$P\in (M_k\otimes \cS_\bA)^\perp=M_k\otimes \cS_\bA^\perp$. Then
\begin{equation*} \tr(D^*C)=\tr((D_{\bA,\bB}^*+P^*)C)=\tr(D_{\bA,\bB}^*C)
=s_{\bA,\bB}(C),\quad C\in \cS_\bA,
\end{equation*} hence, letting $s\colon M_{kn}\ra\CC$ be the linear functional 
associated to the density matrix $D$, it follows that $s$ is positive and extends 
$s_{\bA,\bB}$. Further, let $\phi_s\colon M_n\ra M_k$ be the linear map associated
to $s$ as in \eqref{e:swinv}. Then $\phi$ is completely positive and extends 
$\phi_{\bA,\bB}$.
\end{proof}

\begin{corollary}\label{c:text}
Under the assumptions and the notation 
of Theorem~\ref{t:gen}, suppose that one (hence both) 
of the equivalent conditions \emph{(i)} and \emph{(ii)} holds. 
Then, the formula
\begin{equation} \label{e:densa} \phi(A)=\bigl[ \tr(D_{\bA,\bB}+P)
(E_{i,j}^{(k)}\otimes A)\bigr]_{i,j=1}^k,\quad A\in M_n,
\end{equation}
establishes an affine isomorphism between the closed convex sets
\begin{equation}\label{e:cab} 
\cC_{\bA,\bB}:=\{\phi\in\mathrm{CP}(M_n,M_k)\mid \phi(A_\nu)=B_\nu,
\mbox{ for all }\nu=1,\ldots,N\},
\end{equation}
and
\begin{equation}\label{e:pab} 
\cP_{\bA,\bB}:=\{P\in (M_k\otimes \cS_\bA^\perp)^\mathrm{h}
\mid  P\geq -D_{\bA,\bB}\}.
\end{equation}
\end{corollary}

\begin{proof} It is clear that both sets $\cC_{\bA,\bB}$ and $\cP_{\bA,\bB}$
are closed and convex.

The fact that the formula \eqref{e:densa} establishes an
affine isomorphism between these two convex sets follows, on the one hand, 
from the affine isomorphism properties of the Smith-Ward functional and of the 
density matrix and, on the other hand, from the proof of Theorem~\ref{t:gen}.
\end{proof}

\begin{remarks}\label{r:texta}
Let the assumptions and the notation of Theorem~\ref{t:gen} hold.

\nr{1} In order for the set $\cC_{\bA,\bB}$ to be nonempty, a necessary condition is, 
clearly, that for arbitrary $\nu=1,\ldots,N$, if $A_\nu$ is 
semidefinite, then $B_\nu$ is semidefinite of the same type.

\nr{2} If the density matrix $D_{\bA,\bB}$ is positive semidefinite, as a consequence 
of Corollary~\ref{c:text}, the set 
$\cC_{\bA,\bB}$ is nonempty, more precisely, the map $\phi\colon M_n\ra M_k$ 
defined by
\begin{equation}\label{e:fidab} 
\phi(A)=\bigl[\tr(D_{\bA,\bB}(E_{i,j}^{(k)}\otimes A))\bigr]_{i,j=1}^n,
\quad A\in M_n,
\end{equation} is completely positive and $\phi(A_\nu)=B_\nu$ for all 
$\nu=1,\ldots,N$. We stress the fact that this sufficient condition is, in general, not 
necessary, see Examples~\ref{ex:pp}.

\nr{3} According to Corollary 3.2 in \cite{LiPoon},
for any $A\in M_n^+$ and $B\in M_k^+$ there exists 
$\phi\in \mathrm{CP}(M_n,M_k)$ such that $\phi(A)=B$.
This can be obtained, in our setting, by observing that, in this case, 
$D_{A,B}=B^T\otimes A$ is positive semidefinite and then apply the previous 
statement.
\end{remarks}

The following theorem considers the special case when the $*$-space $\cS_\bA$
is generated by its positive cone $\cS_\bA^+$. This assumption, for example, 
becomes natural if we are looking for solutions $\phi$ 
of the Interpolation Problem that are unital, that is, $\phi(I_n)=I_k$, or if we assume
that the data $\bA$ and $\cB$ consist of quantum states.
The equivalence of assertions (1)
and (2), which is based on Arveson's Hahn-Banach Theorem, has been also 
observed in 
a different setting but equivalent formulation by A.~Jen\v cov\'a, cf.\ Theorem 1 in 
\cite{Jencova}, and by T.~Heinosaari, M.A.~Jivulescu, D.~Reeb, M.M.~Wolf, 
cf.\ Theorem 4 and Corollary 2 in \cite{HJRW} 
(our Corollary~\ref{c:pd} and Corollary~2 in \cite{HJRW} explains that the
two cited theorems are actually equivalent).

\begin{theorem}\label{t:opsys} With the assumptions and the notation as in 
Theorem~\ref{t:gen} assume, in addition, that \emph{(a3)} holds as well.
The following assertions are equivalent:

\nr{1} There exists $\phi\in\mathrm{CP}(M_n\, M_k)$ such that 
$\phi(A_\nu)=B_\nu$ for all $\nu=1,\ldots,N$.

\nr{2} The linear map $\phi_{\bA,\bB}$ defined at \eqref{e:fia} is $k$-positive.

\nr{3} The linear functional 
$s_{\bA,\bB}\colon M_k\otimes \cS_\bA\ra\CC$ defined by \eqref{e:es} 
is positive.

\nr{4} The affine space $D_{\bA,\bB}+M_k\otimes\cS_\bA^\perp$ contains at least 
one positive semidefinite matrix.
\end{theorem}

\begin{proof} (1)$\Ra$(2). 
Let $\phi\colon M_n\ra M_k$ be a completely positive map such that 
$\phi(A_\nu)=B_\nu$ for all $\nu=1,\ldots,N$. Then 
$\phi|\cS_\bA\colon \cS_\bA\ra M_k$
is completely positive, in the sense specified at the beginning of 
Subsection~\ref{ss:sw}, that is, $\phi_{\bA,\bB}=\phi|\cS_A$ is completely positive, 
in particular $k$-positive.

(2)$\Ra$(3). Assume that $\phi_{\bA,\bB}$ is $k$-positive. 
With notation as in \eqref{e:e}, a moment of 
thought shows that, for each $i,j=1,\ldots,N$ and each $\nu=1,\ldots,N$, we have
\begin{equation*}\langle (I_k\otimes\phi_{\bA,\bB})(E_{i,j}^{(k)}\otimes A_\nu)e^{(k)},
e^{(k)}\rangle_{\CC^{k^2}}=\phi_{\bA,\bB}(A_\nu)=B_\nu=
s_{\bA,\bB}(E_{i,j}^{(k)}\otimes A_\nu),
\end{equation*} hence
\begin{equation}\label{e:sifi} \langle (I_k\otimes\phi_{\bA,\bB})(C)e^{(k)},e^{(k)}
\rangle_{\CC^{k^2}}=s_{\bA,\bB}(C),\quad C\in M_k\otimes \cS_\bA,
\end{equation} and, consequently, $s_{\bA,\bB}$ maps any positive semidefinite 
matrix from $M_k\otimes \cS_\bA$ to a nonnegative number.

(3)$\Ra$(1). Assume that the linear functional $s_{\bA,\bB}\colon 
M_k\otimes \cS_\bA\ra\CC$ defined by \eqref{e:es} is positive, in the sense that 
it maps any positive element in $M_k\otimes \cS_\bA=M_k(\cS_\bA)$ into $\RR_+$.
By Arveson's Hahn-Banach Theorem \cite{Arveson}, see the implication 
(c)$\Ra$(d) in Theorem~\ref{t:sw} and the argument provided there, there exists
a completely positive map $\widetilde\phi\colon M_k\ra M_n$ extending 
$\phi_{\bA,\bB}$, hence $\widetilde\phi$ satisfies the same interpolation 
constraints as $\phi_{\bA,\bB}$.

(1)$\Leftrightarrow$(4). Proven in Theorem~\ref{t:gen} 
\end{proof}

\begin{remark}\label{r:param} Under the assumptions and notation as in 
Theorem~\ref{t:opsys}, if $\cS_\bA$ contains a positive definite matrix, then 
the set $\cC_{\bA,\bB}$ is convex and compact, 
see Remark~\ref{r:compact}. Then the set $\cP_{\bA,\bB}$, 
see Corollary~\ref{c:text}, is convex and compact as well.
\end{remark}

\begin{corollary}\label{c:ns} 
If the $*$-subspace $\cS_\bA$ is an algebra, then the set $\cC_{\bA,\bB}$ is
nonempty if and only if $D_{\bA,\bB}$ is positive semidefinite, more precisely,
in this case \eqref{e:fidab} provides a solution $\phi\in\cC_{\bA,\bB}$ of
the Interpolation Problem.
\end{corollary}

\begin{proof} This is a consequence of Theorem~\ref{t:opsys}, the second 
statement  of Remark~\ref{r:hermitian}, and the second statement of
Remark~\ref{r:texta}.
\end{proof}

Example~\ref{ex:pp}.(2) shows that the statement in the previous corollary may
be true without the assumption that the $*$-space $\cS_\bA$ is an algebra.

\begin{remark}\emph{Trace Preserving.}\label{r:ti}
Recall that a linear map $\phi\colon M_n\ra M_k$ is \emph{trace preserving} if 
$\tr(\phi(A))=\tr(A)$ for all $A\in M_n$. With the notation as in Theorem~\ref{t:gen}, 
let
\begin{equation}\label{e:qab} \cQ_{\bA,\bB}:=\{\phi\in \cC_{\bA,\bB}\mid\phi\mbox { is 
trace preserving}\},
\end{equation}
and we want to determine, with respect to the affine isomorphism established 
in Corollary~\ref{c:text}, how the corresponding parameterising subset 
$\cP_{\bA,\bB}$ can be singled out and, implicitly, to get a characterisation of the 
solvability of the  
Interpolation Problem for quantum channels. So let $P=[p_{(i,l),(j,m)}]$ be an 
arbitrary matrix in $\cP_{\bA,\bB}$, where $(i,l)=(i-1)n+l$ and $(j,m)=(j-1)n+m$, 
for $i,j=1,\ldots,k$ and $l,m=1,\ldots,n$, equivalently, in tensor notation,
\begin{equation*} P=\sum_{i,j=1}^k\sum_{l,m=1}^n p_{(i,l),(j,m)} E_{i,j}^{(k)}
\otimes E_{(l,m)}^{(n)}.
\end{equation*}
In view of \eqref{e:densa}, a map $\phi\in\cC_{\bA,\bB}$ is trace invariant if and only 
if
\begin{equation}\label{e:tinv} \tr(\phi(E_{l,m}^{(n)}))=\tr(E_{l,m}^{(n)})
=\delta_{l,m},\quad l,m=1,\ldots,n,
\end{equation} which, taking into account of Proposition~\ref{p:dens}, is equivalent 
with the conjunction of the following affine constraints
\begin{equation}\label{e:tinva} \sum_{i=1}^k p_{(i,m),(i,l)}
=\delta_{l,m}-\sum_{\nu=1}^N a_{m,l,\nu}\bigl(\sum_{i=1}^k d_{i,i,\nu}\bigr),
\quad l,m=1,\ldots,n.
\end{equation}
\end{remark}

\begin{remark}\label{r:zerotrace} Assume that $\cS_\bA$ is an operator system. 
By Theorem~\ref{t:gen}, if the Interpolation Problem has a solution then
there exists a positive semidefinite matrix $\widetilde D$ in 
$D_{\bA,\bB}+M_k\otimes \cS_\bA^\perp$ hence,
by Lemma~\ref{l:orta} we have $0\leq \tr(\widetilde D)
=\tr(D_{\bA,\bB})$. Therefore, under these assumptions, 
a necessary condition of solvability of the Interpolation Problem is
$\tr(D_{\bA,\bB})\geq 0$.
\end{remark}

\subsection{Orthonormalisation of the Input Data.}\label{ss:oid}
Theorem~\ref{t:gen}
gives the necessary and sufficient
condition of solvability of the Interpolation Problem in terms of the density 
matrix $D_{\bA,\bB}$ but, in order to precisely get it one might solve the system 
of linear equations \eqref{e:dem}, with the Gramian matrix
$[\tr(A_\mu^* A_\nu)]_{\mu,\nu}$ as the principal matrix of the system. If
the matrices 
$A_1,\ldots,A_N$ are mutually orthogonal with respect to the Hilbert-Schmidt 
inner product, this Gramian matrix is just the identity matrix $I_N$.
Observe that, if this is not the case, the Gram-Schmidt 
orthonormalisation algorithm yields an orthonormal system of matrices that 
preserves the assumptions (a1)--(a3), due to the fact that the trace of a product of 
two Hermitian matrices is always real. More precisely, if $A_1',\ldots,A_N'$ is the 
Gram-Schmidt orthogonalisation of the sequence of linearly independent Hermitian 
matrices $A_1,\ldots,A_N$ then
\begin{equation*} A_1'=\frac{1}{\sqrt{\tr(A_1^2)}} A_1,\end{equation*}
\begin{equation*}U_{\nu+1}=A_{\nu+1}-\sum_{\mu=1}^\nu \tr(A'_\mu A_\mu)A_\mu,
\quad A_{\nu+1}'=\frac{1}{\sqrt{\tr(U_{\nu+1}^2)}}U_{\nu+1},\quad \nu=1,\ldots,N-1.
\end{equation*}
Then, we can change, accordingly, the sequence $B_1,\ldots,B_N$ 
to $B'_1,\ldots,B'_N$
\begin{equation*} B_1'=\frac{1}{\sqrt{\tr(A_1^2)}} B_1,\quad B_{\nu+1}'=
\frac{1}{\sqrt{\tr(U_{\nu+1}^2)}}\bigl(B_{\nu+1} -\sum_{\mu=1}^\nu 
\tr(A'_\mu A_\mu)B_\mu\big),\quad \nu=1,\ldots,N-1,
\end{equation*} and observe that a linear map $\phi\colon M_n\ra M_k$ satisfies
the constraints $\phi(A_\nu)=B_\nu$, $\nu=1,\ldots,N$, if and only if 
$\phi(A_\nu')=B_\nu'$, $\nu=1,\ldots,N$.
Therefore, without loss of generality, we can replace the assumption (a2) with 
the assumption
\begin{itemize}
\item[(a2')] \emph{The set of matrices $\{A_1,\ldots,A_N\}$ is orthonormal with
respect to the Hilbert-Schmidt inner product, that is, $\tr(A_\mu A_\nu)
=\delta_{\mu,\nu}$ for all $\mu,\nu=1,\ldots,N$.}
\end{itemize}

\begin{lemma}\label{l:pos} Under the assumptions \emph{(a1)} and \emph{(a2')}, 
the density matrix $D_{\bA,\bB}$ of the linear functional $s_{\bA,\bB}$ 
defined at \eqref{e:es} is
\begin{equation}\label{e:dedoi} 
D_{\bA,\bB}=\sum_{\nu=1}^N B_\nu^T \otimes A_\nu.
\end{equation}
\end{lemma}

\begin{proof} Under the assumption (a2'), the Gramian matrix of $A_1,\ldots,A_N$ 
is the identity matrix $I_N$ hence the system of linear equations \eqref{e:dem} 
is simply solvable as $d_{i,j,\nu}=\overline{b}_{i,j,\nu}=b_{j,i,\nu}$ 
for all $i,j=1,\ldots,k$ and all $\nu=1,\ldots,N$, where we have taken into account 
that $B_\nu$ are all Hermitian matrices. Thus, by \eqref{e:de}, we have
\begin{equation*}
D_{\bA,\bB}=\sum_{\nu=1}^N \sum_{i,j=1}^k b_{j,i,\nu} E_{i,j}^{(k)}\otimes A_\nu
=\sum_{\nu=1}^N B_\nu^T\otimes A_\nu.\qedhere
\end{equation*}
\end{proof}

Note that under the assumptions (a1) and (a2') we can always 
find an orthonormal basis $A_1,\ldots,A_N,A_{N+1},\ldots,A_{n^2}$ of $M_n$, with
respect to the Hilbert-Schmidt inner product, whose first elements are exactly the 
elements of the input data $\bA$ and such that all its matrices are Hermitian. 
Indeed, this basically follows from the fact that $\cS_\bA^\perp$ is a $*$-space, and 
the remark we made before on the Gram-Schmidt orthonormalisation of a sequence 
of linearly independent Hermitian matrices.

\begin{theorem}\label{t:ort}
Assume that the data $A_1,\ldots,A_N$ and $B_1,\ldots,B_N$ satisfy the 
assumptions \emph{(a1)} and \emph{(a2')}. Let $A_{N+1},\ldots,A_{n^2}$ be a
sequence of Hermitian matrices in $M_n$ such that $A_1,\ldots,A_{n^2}$ is an
orthonormal basis of $M_n$ with respect to the Hilbert-Schmidt inner product.
The following assertions are equivalent:

\nr{1} There exists $\phi\in\mathrm{CP}(M_n,M_k)$ such that 
$\phi(A_\nu)=B_\nu$ for all $\nu=1,\ldots,N$.

\nr{2} There exist numbers $p_{i,j,\nu}$, $i,j=1,\ldots,k$ and $\nu=N+1,\ldots,n^2$, 
such that $p_{j,i,\nu}=\overline{p}_{i,j,\nu}$ and
\begin{equation}\label{e:cenu} \sum_{\nu=1}^N B_\nu^T \otimes A_\nu
+\sum_{i,j=1}^k \sum_{\nu=N+1}^{n^2} p_{i,j,\nu}E_{i,j}^{(k)}\otimes A_\nu \geq 0.
\end{equation}
\end{theorem}

\begin{proof} We use Theorem~\ref{t:gen}, by 
means of Lemma~\ref{l:pos} and Lemma~\ref{l:orta}, taking into account that in 
order to get a positive semidefinite matrix in the affine space
$D_{\bA,\bB}+M_k\otimes\cS_{\bA}^\perp$ we actually look for a Hermitian 
element $P\in M_k\otimes\cS_{\bA}^\perp$, more precisely
\begin{equation*}P=\sum_{i,j=1}^{k}\sum_{\nu=N+1}^{n^2} p_{i,j,\nu} E_{i,j}^{(k)}
\otimes A_\nu,
\end{equation*} such that $D_{\bA,\bB}+ P\geq 0$.
\end{proof}

\begin{remarks}\label{r:suf} 
Assume that the data $A_1,\ldots,A_N$ and $B_1,\ldots,B_N$ satisfy the 
assumptions (a1) and (a2'). 

\noindent (i) If the set $\cC_{\bA,\bB}$ is nonempty,  then, as a consequence 
of Lemma~\ref{l:pos} and Remark~\ref{r:texta}, the following conditions are 
necessary:

\begin{itemize}
\item[(a)] For arbitrary $\nu=1,\ldots,N$, if $A_\nu$ is semidefinite then $B_\nu$ is
semidefinite of the same type.
\item[(b)] If $\cS_\bA$ is an operator system, then $\sum_{\nu=1}^N \tr(B_\nu)\tr(A_\nu)\geq 0$.
\end{itemize}

\noindent (ii) On the other hand, from Lemma~\ref{l:pos} and 
Remark~\ref{r:texta}.(2), if
\begin{equation}\label{e:suf} \sum_{\nu=1}^N B_\nu^T\otimes A_\nu \geq 0,
\end{equation} then the linear map $\phi\colon M_n\ra M_k$ defined by
\begin{equation}\label{e:cp} \phi(C)=\bigl[\sum_{\nu=1}^N 
b_{i,j,\nu} \tr(A_\nu C)\bigr]_{i,j=1}^k,\quad C\in M_n,
\end{equation}
where $b_{i,j,\nu}$ are the entries of the matrix $B_\nu$, see \eqref{e:be}, 
is completely positive and satisfies the interpolation constraints $\phi(A_\nu)=B_\nu$ 
for all $\nu=1,\ldots,N$.
\end{remarks}

\subsection{A Single Interpolation Pair.}\label{ss:sip}
For fixed $n,k\in\NN$, consider completely positive maps 
$\phi\colon M_n\ra M_k$ in the minimal Kraus representation, that is, 
$\phi(A)=V^*AV$, for some $V\in M_{k,n}$ and all
$A\in M_n$. By Proposition~\ref{p:minimal}, 
this corresponds to the case when the rank of the Choi matrix $\Phi$ of $\phi$ 
is $1$. For given Hermitian matrices $A\in M_n$ and $B\in M_k$, we are interested 
to determine under which conditions on $A$ and $B$ there exists a completely 
positive maps $\phi$ in the minimal Kraus representation such that $\phi(A)=B$. 

If $A$ is a Hermitian $n\times n$ matrix, we consider the decomposition 
$A=|A|^{1/2} S_A |A|^{1/2}$, where $|A|=(A^*A)^{1/2}$ is its absolute value, 
while $S_A=\sgn(A)$ is a Hermitian partial isometry, where $\sgn$ is the usual sign 
function: $\sgn(t)=1$ for $t>0$, $\sgn(t)=-1$ for $t<0$, and $\sgn(0)=0$. Note that, 
with this notation, $A=S_A|A|$ is the polar decomposition of $A$. Let 
$\cH_A=\CC^n\ominus \ker(A)$ and, further, consider the decomposition $\cH_A=
\cH_A^+\oplus \cH_A^-$, where $\cH_A^\pm$ is the spectral subspace of $S_A$ 
(and of $A$, as well) corresponding, respectively, to the eigenvalue $\pm 1$. Then,
with respect to the decomposition
\begin{equation}\label{e:decoma} \CC^n=\cH_A^+\oplus \cH_A^-\oplus \ker(A),
\end{equation}
we have
\begin{equation}\label{e:rep} A=\left[\begin{array}{lll}
A_+ & 0 & 0 \\
0 & -A_- & 0 \\
0 & 0 & 0 
\end{array}\right],\quad
S_A=\left[\begin{array}{lll}
I_A^+ & 0 & 0 \\
0 & -I_A^- & 0 \\
0 & 0 & 0 
\end{array}\right],
\end{equation}
where $A_\pm$ act in $\cH_A^\pm$, respectively, are positive operators, and 
$I_A^\pm$ are the identity operators in $\cH_A^\pm$, respectively.

With this notation, we consider the \emph{signatures} 
$\kappa_\pm(A)=\dim(\cH_A^\pm)=\rank(A_\pm)$ and $\kappa_0(A)=\dim(\ker(A))$.
The triple $(\kappa_-(A),\kappa_0(A),\kappa_+(A))$ is called the \emph{inertia} of 
$A$. Note that $\kappa_\pm(A)$ is the number of positive/negative eigenvalues of
the matrix $A$, counted with their multiplicities, 
as well as the \emph{number of 
negative/positive squares} of the quadratic form 
$\CC^n\ni x\mapsto \langle Ax,x\rangle$. In this respect, the space $\CC^n$ has 
natural structure of \emph{indefinite inner product} with respect to 
\begin{equation} [x,y]_A=\langle Ax,y\rangle,\quad x,y\in\CC^n.
\end{equation}
Then, $\kappa_\pm(A)$ coincides with the dimension of any $A$-maximal 
positive/negative subspace: here, a subspace $\cL\in\CC^n$ is called 
\emph{positive} if $[x,x]_A>0$ for all nonnull $x\in \cL$.

\begin{lemma}\label{l:oneherminter} 
Let $A\in M_n$ and $B\in M_k$ be two Hermitian matrices. Then, 
there exists a completely positive map $\phi$ with minimal Kraus (Choi) rank equal 
to $1$ and such that $\phi(A)=B$ if and only if $\kappa_\pm (B)\leq \kappa_\pm(A)$.
\end{lemma}

\begin{proof} Assume that $B=V^*AV$ for some $V\in M_{k,n}$ and note that for all 
nonnull $x\in \cH_B^+$ we have
\begin{equation*}0< [x,x]_B=\langle Bx,x\rangle=\langle V^*AVx,x\rangle=\langle 
AVx,Vx\rangle=[Vx,Vx]_A,
\end{equation*} hence, the subspace $V\cH_B^+$ is $A$-positive, and this implies 
that $\kappa_+(B)\leq \kappa_+(A)$. Similarly, we have 
$\kappa_-(B)\leq \kappa_-(A)$.

Conversely, let assume that $\kappa_\pm(B)\leq \kappa_\pm(A)$, that is, 
$\dim(\cH_B^\pm)\leq \dim(\cH_A^\pm)$, and hence there exists isometric operators 
$J_\pm\colon \cH_B^\pm\ra \cH_A^\pm$. In addition to the decomposition 
\eqref{e:decoma} of $\CC^n$ with respect to $A$, we consider the like decomposition
of $\CC^k$ with respect to $B$
\begin{equation}\label{e:decomb} 
\CC^k=\cH_B^+\oplus \cH_B^-\oplus \ker(B),
\end{equation} and with respect to it, the block-matrix representation of $B$ similar
to \eqref{e:rep}.
 Then, with respect to \eqref{e:decoma} and \eqref{e:decomb} define 
$V\in M_{k,n}$ by
\begin{equation}\label{e:ve} V=\left[\begin{array}{lll}
B_+^{1/2}J_+A_+^{-1/2} & 0 & 0 \\
0 & B_-^{1/2}J_-A_-^{-1/2} & 0 \\
0 & 0 & 0 
\end{array}\right],
\end{equation}
and then, a simple calculation shows that $V^*AV=B$.
\end{proof}

\begin{theorem}\label{t:oneherminterp} 
Let $A\in M_n$ and $B\in M_k$ be two Hermitian matrices. Then the following 
assertions are equivalent:

\nr{i}There exists a 
completely positive map $\phi\colon M_n\ra M_k$ such that $\phi(A)=B$.

\nr{ii} If $A$ is semidefinite, then $B$ is semidefinite of the same type.

\nr{iii} There exists $m\in \NN$ such that
\begin{equation}\label{e:em} \kappa_\pm(B)\leq 
m\,\kappa_\pm(A).
\end{equation} 

In addition,
the minimal Kraus (Choi) rank $r$ of a completely positive map 
$\phi\colon M_n\ra M_k$ such that $\phi(A)=B$ is 
\begin{equation}\label{e:er} r=\min\{ m\in\NN\mid \kappa_\pm(B)\leq 
m\,\kappa_\pm(A)\}.
\end{equation}\end{theorem}

\begin{proof} It takes only a moment of thought to see that the assertions (ii) and (iii) 
are equivalent. Thus, it remains to prove that the assertions (i) and (iii) are 
equivalent. Assuming that there exists $m\in\NN$ satisfying \eqref{e:em}, 
let $r\in\NN$ be defined as in \eqref{e:er}. Then there exist Hermitian 
matrices $B_1,B_2,\ldots, B_r\in M_k$ such that 
$\kappa_\pm(B_j)\leq\kappa_\pm(A)$ for all $j=1,\ldots,r$. By 
Lemma~\ref{l:oneherminter} there exist $V_1,V_2,\ldots,V_r\in M_{k,n}$ 
such that $V_j^*AV_j=B_j$ for all $j=1,\ldots,r$. Then, letting 
$\phi=\sum_{j=1}^r V_j^*\cdot V_j\colon M_n\ra M_k$ we obtain a completely 
positive map such that $\phi(A)=B$.

On the other hand, if $V_1,V_2,\ldots,V_m\in M_{k,n}$ are such that 
$\sum_{j=1}^m V_j^*AV_j=B$ then for each $j=1,\ldots,m$ we have 
$k_\pm(V_j^*AV_j)\leq \kappa_\pm(A)$, hence 
$\kappa_\pm(B)\leq \sum_{j=1}^m \kappa_\pm(V_j^*AV_j)\leq m\kappa_\pm(A)$, 
hence $r\leq m$.
\end{proof}

Note that Theorem~\ref{t:oneherminterp} provides one more (different) argument for 
Corollary~3.2 in \cite{LiPoon}, and different from the argument given in 
Remark~\ref{r:texta}.(3) as well.

\end{document}